\documentclass[10pt]{amsart}

\usepackage[latin2]{inputenc}
\usepackage{amsmath}
\usepackage{graphicx}
\usepackage{amssymb}
\usepackage{esint}
\usepackage{tikz}
\usepackage{xxcolor}
\usepackage{floatrow}
\usepackage{color}
\usepackage{amsthm}
\usepackage{epsfig}
\usepackage[english]{babel}

\newtheorem{theorem}{Theorem}
\newtheorem{proposition}[theorem]{Proposition}
\newtheorem{lemma}[theorem]{Lemma}
\newtheorem{corollary}[theorem]{Corollary}

\newtheorem*{theorem*}{Theorem}

\def\Xint#1{\mathchoice
{\XXint\displaystyle\textstyle{#1}}%
{\XXint\textstyle\scriptstyle{#1}}%
{\XXint\scriptstyle\scriptscriptstyle{#1}}%
{\XXint\scriptscriptstyle\scriptscriptstyle{#1}}%
\!\int}
\def\XXint#1#2#3{{\setbox0=\hbox{$#1{#2#3}{\int}$ }
\vcenter{\hbox{$#2#3$ }}\kern-.6\wd0}}

\def\dashint{\Xint-}

\allowdisplaybreaks[2]

\newcommand{\Id}{\operatorname{Id}} 
 
\newcommand{\divv}{\operatorname{\nabla \cdot}}

\newcommand{\Exc}{\operatorname{Exc}}
\newcommand{\av}{-\hspace{-2.4ex}\int}

\renewcommand{\P}[1]{\mathcal{P}^{#1}}
\newcommand{\Pa}[1]{\mathcal{P}^{#1}_{a_{hom}}}

\definecolor{Yellow}{rgb}{0.95,0.9,0.0} 
\definecolor{Red}{rgb}{0.8,0.1,0.1}
\definecolor{Green}{rgb}{0.1,0.65,0.2}
\definecolor{Blue}{rgb}{0.1,0.1,0.8}
\definecolor{Purple}{rgb}{0.7,0.1,0.7}
\definecolor{Grey}{rgb}{0.6,0.6,0.6}

\begin{document}

\newcommand{\InternalRemark}[1]{{\color{red}#1}}


\title[Higher regularity for random elliptic operators]
{A higher-order large-scale regularity theory for random elliptic operators}

\author{Julian Fischer and Felix Otto}
\address{Max Planck Institute for Mathematics in the Sciences,
Inselstrasse 22, 04103 Leipzig, Germany, E-Mail:
julian.fischer@mis.mpg.de}
\address{Max Planck Institute for Mathematics in the Sciences,
Inselstrasse 22, 04103 Leipzig, Germany, E-Mail:
otto@mis.mpg.de}
\begin{abstract}
We develop a large-scale regularity theory of higher order for divergence-form
elliptic equations with heterogeneous coefficient fields $a$ in the context of
stochastic homogenization.
The large-scale regularity of $a$-harmonic functions is encoded by Liouville
principles:
The space of $a$-harmonic functions that grow at most like a polynomial of
degree $k$ has the same dimension as in the constant-coefficient case. This
result can be seen as the qualitative side of a large-scale
$C^{k,\alpha}$-regularity theory, which in the present work is developed in the
form of a corresponding $C^{k,\alpha}$-``excess decay'' estimate:
For a given $a$-harmonic function
$u$ on a ball $B_R$, its energy distance on some ball $B_r$
to the above space of $a$-harmonic
functions that grow at most like a polynomial of degree $k$ 
has the natural decay in the radius $r$ above some minimal radius $r_0$.

Though motivated by stochastic homogenization, the contribution of this paper is
of purely deterministic nature:
We work under the assumption that for the given realization $a$ of the
coefficient field, the couple $(\phi,\sigma)$ of scalar and vector potentials
of the harmonic coordinates, where $\phi$ is the usual corrector, grows
sublinearly in a mildly quantified way.
We then construct ``$k$th-order correctors'' and thereby the space of
$a$-harmonic functions that grow at most like a polynomial of degree $k$,
establish the above excess decay and then the corresponding Liouville
principle.
\end{abstract}


\maketitle

\section{Introduction}

We are interested in the regularity of harmonic functions $u$ associated with a
uniformly elliptic coefficient field $a$ in $d$ space dimensions (by which we
understand a tensor field satisfying $\lambda |\xi|^2\leq \xi\cdot a\xi$ and
$|a\xi|\leq |\xi|$ for some $\lambda>0$ and any $\xi\in \mathbb{R}^d$) via the
divergence-form equation
\begin{equation}
\label{AHarmonic}
-\divv a \nabla u=0.
\end{equation}
Without continuity assumptions, the local regularity of (weak finite-energy)
solutions can be rather low, in particular in case of systems (cf. e.g.
\cite[Example 3]{PiccininiSpagnolo} for the scalar case and \cite[Section
9.1.1]{GiaquintaMartinazzi} for De Giorgi's celebrated counterexample in the
systems case).
Because of their homogeneity, the same examples show that
even when the coefficients are uniformly locally smooth, the {\it large-scale}
behavior of $a$-harmonic functions can be very different from the constant
coefficient, that is, Euclidean case; cf.\ e.g.\ Proposition
\ref{SmoothLiouvilleCounterexample} in the appendix below.
Large-scale regularity is most compactly encoded in a Liouville statement of the
following form:
The space of $a$-harmonic functions $u$ of growth not larger than $|x|^k$
has the same dimension as in the constant-coefficient case, where
the space is spanned by spherical harmonics up to order $k$. 
Because of the above-mentioned counterexamples, 
such Liouville statements may fail for uniformly elliptic coefficient fields:
For example, in the case of systems, there are non-constant harmonic maps
that decay to zero at infinity.

The question whether this situation generically improves for certain {\it
ensembles} of coefficient fields, namely stationary and ergodic ensembles as in
stochastic homogenization, seems to have first been 
phrased and partially answered in
\cite[Chapter 6 and Theorem 3]{BenjaminiCopinKozmaYadin}
in the context of random walks in random environments:
Under the mere assumption of ergodicity and stationarity,
sublinearly growing $a$-harmonic functions are almost surely constant.
The argument is limited to the scalar case, but can deal with non-uniformly
elliptic cases as percolation.

Motivated by error estimates in stochastic homogenization, the topic of a 
regularity theory for random elliptic operators
was independently addressed in a more quantitative way in 
\cite{MarahrensOtto}.
In Corollary 4 of that paper, for any $\alpha<1$, a large-scale
$C^{0,\alpha}$-inner regularity estimate for $a$-harmonic functions has been
established, with a random constant of finite algebraic moments --- however
under stronger assumptions on the ergodicity, namely a finite spectral gap
w.\ r.\ t.\ Glauber dynamics in the case of a discrete medium.

A major step forward constitutes \cite{ArmstrongSmart}, where the above result
was improved to a large-scale $C^{0,1}$-inner regularity estimate even in case
of (symmetric) systems, by showing that the approach of \cite{AvellanedaLin}
for obtaining (large-scale) regularity of $a$-harmonic
maps, which itself is based on a Campanato-type iteration, 
can be extended from periodic to random coefficient fields. Under a strong
assumption of ergodicity, namely that of a finite range of dependence, 
optimal exponential moments for the random constant are obtained.

This work motivated \cite{GloriaNeukammOtto}, which in turn is the basis for
the present paper.
In that work, another tool from periodic homogenization, namely the {\it
vector} potential $\sigma$ for the harmonic coordinates (next to the well-known
scalar potential $\phi$, also called the corrector), was transferred to the
random case, see (\ref{Equationq}) and (\ref{EquationSigma})  for the
characterizing properties. This allowed to establish
a $C^{1,\alpha}$-Liouville theorem, meaning that the space of sub-quadratically
growing $a$-harmonic functions is almost surely spanned by the the constants
and the $d$ $a$-harmonic coordinates $x_i+\phi_i$.
This holds even for non-symmetric systems and
was shown under the mere assumptions of stationarity and ergodicity. More
precisely, it relied on the almost-sure sublinear growth of the couple
$(\phi,\sigma)$ of correctors in the sense of
\begin{align}
\label{EpsilonSmall}
\lim_{r\rightarrow \infty} \varepsilon_{r}=0,
\end{align}
where
\begin{align}
\label{DefinitionEpsilon}
\varepsilon_{r}:=\sup_{R\geq r}
\frac{1}{R}\left(\dashint_{B_R} |\phi|^2+|\sigma|^2 ~dx\right)^{1/2}.
\end{align}
This sublinear growth \eqref{EpsilonSmall} was shown to hold under the
assumptions of stationarity and qualitative ergodicity.
In a second step, large-scale $C^{1,\alpha}$-inner regularity estimates for
$a$-harmonic functions were obtained, where the random constant satisfies a
stretched exponential bound under mild decay assumptions on the spatial
covariance of $a$. In a later version of \cite{GloriaNeukammOtto}, the optimal
stochastic moments for the random constant were obtained.

In the context of non-linear elliptic systems in divergence form,
the result of \cite{ArmstrongSmart} on the large-scale $C^{0,1}$-estimate
was generalized in \cite{ArmstrongMourrat} to non-symmetric coefficients
and well beyond finite range, further confirming that the random large-scale
regularity theory holds under just a mild quantification of ergodicity, like
expressed by standard mixing conditions.

In the present work, we go beyond $C^{1,\alpha}$ and establish
a large-scale $C^{k,\alpha}$-theory in form of a
corresponding excess decay and Liouville result, cf.\ Theorem
\ref{CkalphaExcessDecay} and Corollary \ref{LiouvillePrincipleK}.
This lifts the result of \cite{AvellanedaLin} from the periodic to the random case.
To streamline presentation, we first establish the
$C^{2,\alpha}$-versions of our theorems, cf.\ Theorem
\ref{C2alphaExcessDecay} and Corollary \ref{LiouvillePrinciple}.

Let us clearly state that the contribution of this paper is exclusively on the
deterministic side.
The large-scale regularity is obtained under the assumption that the given
realization $a$ of the coefficient field is such that the corresponding
corrector couple $(\phi,\sigma)$ satisfies the following slight quantification
of (\ref{EpsilonSmall}), namely
\begin{equation}\label{Epsilon2Small}
\lim_{r\rightarrow \infty} \varepsilon_{2,r}=0
\end{equation}
with
\begin{align}
\label{DefinitionEpsilon2}
\varepsilon_{2,r}:=\sum_{m=0}^\infty \min\{1,2^{m+1}/r\} \varepsilon_{2^m}.
\end{align}
Note that \eqref{Epsilon2Small} is equivalent to
$\sum_{m=0}^\infty \varepsilon_{2^m}<\infty$.

In the recent preprint \cite{FischerOtto2}, it is shown
that (\ref{Epsilon2Small}) holds for almost every realization $a$ in case of a
stationary ensemble of coefficient fields under mild quantification of
ergodicity in form of an assumption on a mild decay of correlations of $a$:
More precisely, given a stationary centered tensor-valued Gaussian random
field $\tilde a$ on $\mathbb{R}^d$ and a bounded Lipschitz map
$\Phi:\mathbb{R}^{d\times d}\rightarrow \mathbb{R}^{d\times d}$ taking values
in the set of $\lambda$-uniformly elliptic tensors, the coefficient field
$a:=\Phi(\tilde a)$ almost surely admits correctors with the property
\eqref{Epsilon2Small} assuming just decay of correlations in the sense $|\langle
\tilde a(x)\tilde a(y) \rangle| \leq C|x-y|^{-\beta}$ for some $C>0$ and some
$\beta\in (0,c(d,\lambda))$ (where $\langle\cdot\rangle$ denotes the expectation).
Note that under the assumption of a spectral gap for the ensemble, as far as the
corrector $\phi$ is concerned (but not the ``vector potential'' $\sigma$), an
estimate like (\ref{Epsilon2Small}) could also be deduced to hold almost surely
from \cite[Proposition 2]{GloriaOttoAP2011}, modulo the passage from a discrete to a
continuum medium.

The key building block for this large-scale $C^{k,\alpha}$-theory is the space
of $a$-harmonic functions that grow at most like a polynomial of degree $k$ at
infinity.
Proposition \ref{ExistenceCorrectorsPolynomialK} and Corollary
\ref{LiouvillePrincipleK} imply that under our assumption (\ref{Epsilon2Small})
this space has the same dimension as in the Euclidean case -- e.g.\ for $k=2$
the space of $a$-harmonic functions that grow at most quadratically is spanned by
$1+d+\frac{d(d+1)}{2}$ \mbox{maps -- ,} which partially answers the
question in \cite[Chapter 6]{BenjaminiCopinKozmaYadin}.
The $k$th-order excess (\ref{kthOrderExcessFirst}), by the decay of which
we encode the $C^{k,\alpha}$-theory, measures the distance to this space in
terms of the averaged squared gradient.
As our construction shows, there is a one-to-one correspondence between the
asymptotic behavior of functions in this space and $a_{hom}$-harmonic
polynomials of degree $k$.
However, there is no natural one-to-one correspondence between elements of this
space and $k$th-order $a_{hom}$-harmonic polynomials.

Before stating our results, let us recall the definition of the correctors
$(\phi,\sigma)$.
The corrector $\phi_i$ satisfies the equation
\begin{align}
\label{EquationCorrector}
-\divv a (e_i+\nabla \phi_i)=0.
\end{align}
The flux correction $q_{ij}$ is defined as
\begin{align}
\label{Equationq}
q_i := a(e_i+\nabla \phi_i)-a_{hom}e_i
\end{align}
where $a_{hom}$ is the homogenized tensor, that is, $a_{hom}e_i$ is the expectation
of $a(e_i+\nabla \phi_i)$. In our analysis, we will only use that $a_{hom}$ is
some constant elliptic coefficient.
We introduce the corresponding vector potential
$\sigma_{ijk}$ (antisymmetric in its last two indices) by requiring that
\begin{align}
\label{EquationSigma}
\divv \sigma_{ij} = q_{ij}.
\end{align}
For the actual construction of a $\sigma$ with stationary gradient we refer
to \cite{GloriaNeukammOtto}; in this note, we just use the property (\ref{EquationSigma}).
In the context of periodic homogenization, both the scalar and the
vector potential $\phi$ and $\sigma$ may be chosen to be periodic. In stochastic
homogenization, one cannot always expect to have a stationary $(\phi,\sigma)$
(for instance in $d\le 2$ even in case of finite range) but, as mentioned above,
we expect sublinear growth in the sense of (\ref{Epsilon2Small}) under
mild ergodicity assumptions.

{\bf Notation.} Throughout the paper, we use the Einstein summation convention,
i.e. we implicitly take the sum over an index whenever this index occurs twice.
For example, $b_i \partial_i v$ is an alternative notation for $(b \cdot
\nabla) v$ and $b_i \nabla v_i$ is an alternative notation for $\sum_{i=1}^d
b_i \nabla v_i$.

By $C$ we denote a generic constant whose value may be different in each
appearance of the expression $C$; similarly, by e.g.\ $C(d,\lambda)$ we denote
a generic constant depending only on $d$ and $\lambda$ whose value again may be
different for every use of the expression $C(d,\lambda)$.

By $\mathcal{E}:=\{E\in \mathbb{R}^{d\times d}:
(E_{ij}+E_{ji})(a_{hom})_{ij}=0\}$ we denote the space of matrices $E_{ij}$ for
which $E_{ij} x_i x_j$ is an $a_{hom}$-harmonic second-order polynomial.

The notation $P$ (or $P(x)$) generally refers to a polynomial.
By $\P{k}$, we denote the space of homogeneous polynomials of degree $k$.
By $\Pa{k}$, we denote the space of homogeneous polynomials of
degree $k$ which are $a_{hom}$-harmonic. On the space $\P{k}$, we introduce the
norm $||P||:=\sup_{x\in B_1} |P(x)|$; note that any other norm on this
finite-dimensional space would do as well, since we do not care for
$C(k)$-constants.

\section{Main Results}

The proof of our large-scale $C^{k,\alpha}$ regularity theory relies in an
essential way on the existence of $k$th-order correctors for the homogenization
problem, which enable us to correct $a_{hom}$-harmonic polynomials of degree $k$
by adding a small (in the $L^2$-sense) perturbation.

The ansatz for the deformation of an $a_{hom}$-harmonic polynomial
$P$, homogeneous of degree $k$ (i.e.\ $P\in \Pa{k}$), into an
$a$-harmonic function $u$ with the same growth behavior
is motivated by homogenization: We consider $P$ as
the ``homogenized solution of the problem solved by $u$'', so that we think in
terms of the two-scale expansion $u\approx
P+\phi_k\partial_kP$
and have that the error
$\psi_P:=u-(P+\phi_k\partial_kP)$ satisfies $-\nabla\cdot
a\nabla\psi_P=\nabla\cdot((\phi_ka-\sigma_k)\nabla\partial_kP)$.
In order to construct $u$, we reverse the logic and first construct a solution
$\psi_P$ to the above elliptic equation and then set
$u:=P+\phi_k\partial_kP+\psi_P$.

\begin{theorem}[Existence of higher-order ``correctors for polynomials'']
\label{ExistenceHigherOrderCorrector}
Let $d\geq 2$, $k\geq 2$, and suppose that the corrector $\phi$ and the
flux-correction potential $\sigma$ satisfy the growth assumption
(\ref{Epsilon2Small}).
Let $r_0$ be large enough so that $\varepsilon_{2,r_0}\leq \varepsilon_0$ holds
(the existence of such $r_0$ is ensured by (\ref{Epsilon2Small})), where
$\varepsilon_0=\varepsilon_0(d,k,\lambda)>0$ is a constant defined in the proof
below.
Given any $P\in \P{k}$,
there exists a ``corrector for polynomials'' $\psi_P$ satisfying
\begin{align}
\label{EquationPsiP}
-\divv a\nabla \psi_P =
\divv ((\phi_i a-\sigma_{i}) \nabla \partial_i P)
\end{align}
as well as
\begin{align}
\label{EstimateGrowthHigherOrderCorrector}
\sup_{R\geq r} \frac{1}{R^{k-1}} \left(\dashint_{B_R} |\nabla \psi_P|^2 ~dx
\right)^{1/2}
\leq C(d,k,\lambda) ||P||
\varepsilon_{2,r}
\end{align}
for any $r\geq r_0$.
Moreover, $\psi_P$ depends linearly on $P$.
\end{theorem}

Our $\psi_P$ indeed enable us -- in conjunction with the first-order correctors
$\phi_i$ -- to correct $a_{hom}$-harmonic $k$th-order polynomials.
\begin{proposition}
\label{ExistenceCorrectorsPolynomialK}
Let $d\geq 2$, $k\geq 2$, and let $P\in \Pa{k}$.
Suppose that $\psi_P$ satisfies \eqref{EquationPsiP}.
We then have
\begin{align*}
-\divv a \nabla (P+\phi_i \partial_i P+\psi_P)=0.
\end{align*}
\end{proposition}

Let us now state our $C^{k,\alpha}$ large-scale regularity result.

\begin{theorem}[$C^{k,\alpha}$ large-scale excess-decay estimate]
\label{CkalphaExcessDecay}
Let $d\geq 2$, $k\geq 2$, and suppose that (\ref{Epsilon2Small}) holds. Let $u$
be an $a$-harmonic function. Let $\psi_P\equiv 0$ for linear polynomials $P$
(in order to simplify notation) and let $\psi_P$ be the functions constructed in
Theorem~\ref{ExistenceHigherOrderCorrector} for higher-order polynomials.
Consider the $k$th-order excess
\begin{align}
\label{kthOrderExcessFirst}
\Exc_k&(r):=
\inf_{P_\kappa\in \Pa{\kappa}}
&\dashint_{B_r} \left|\nabla u-\nabla
\sum_{\kappa=1}^k (P_\kappa + \phi_i \partial_i P_\kappa +
\psi_{P_\kappa})
\right|^2 ~dx.
\end{align}
Let $0<\alpha<1$ and let
$r_0$ be large enough so that $\varepsilon_{2,r_0}\leq \varepsilon_0$ holds
(the existence of such $r_0$ is ensured by (\ref{Epsilon2Small})),
where $\varepsilon_0=\varepsilon_0(d,k,\lambda,\alpha)>0$ is a constant defined
in the proof below.
Then for all $r,R\geq r_0$ with $r<R$ the $C^{k,\alpha}$ excess-decay
estimate
\begin{align}
\label{kthOrderExcessDecayFinal}
\Exc_k(r)\leq C(d,k,\lambda,\alpha)\left(\frac{r}{R}\right)^{2(k-1)+2\alpha}
\Exc_k(R)
\end{align}
is satisfied.
\end{theorem}
Our large-scale $C^{k+1,\alpha}$ excess-decay estimate entails the following
$k$th-order Liouville principle.
\begin{corollary}[$k$th-order Liouville principle]
\label{LiouvillePrincipleK}
Let $d\geq 2$, $k\geq 2$, and suppose that the assumption (\ref{Epsilon2Small})
is satisfied.
Then the following property holds: Any $a$-harmonic function $u$
satisfying the growth condition
\begin{align}
\label{GrowthConditionInkLiouville}
\liminf_{r\rightarrow \infty} \frac{1}{r^{k+1}}
\left(\dashint_{B_r} \left|u\right|^2 ~dx\right)^{1/2}=0
\end{align}
is of the form
\begin{align*}
u=a+b_i (x_i+\phi_i)+\sum_{\kappa=2}^k (P_\kappa+\phi_i \partial_i
P_\kappa+\psi_{P_\kappa})
\end{align*}
with some $a\in \mathbb{R}$, $b\in \mathbb{R}^d$, and $P_\kappa\in
\Pa{\kappa}$ for $2\leq \kappa\leq k$ (i.e. $P_\kappa$ is a homogeneous
$a_{hom}$-harmonic polynomial of degree $\kappa$). Here, the $\psi_{P}$ denote
the higher-order correctors whose existence is guaranteed by Theorem
\ref{ExistenceHigherOrderCorrector}.

In particular, the space of all $a$-harmonic functions satisfying
\eqref{GrowthConditionInkLiouville} has the same dimension as if $a$ was
replaced by a constant coefficient, say $a_{hom}$.
\end{corollary}

The structure of our proofs is as follows:
\begin{itemize}
  \item First, under the assumption that we already have constructed an
  appropriate $k$th-order corrector on a ball $B_R$, we show a $C^{k,\alpha}$
  excess-decay estimate on large scales within this ball for $a$-harmonic
  functions (Lemma \ref{CkalphaExcessDecayGeneral}). This result directly
  implies Theorem \ref{CkalphaExcessDecay} as soon as we have proven the
  existence of a corrector on $\mathbb{R}^d$ (i.e. Theorem
  \ref{ExistenceHigherOrderCorrector}).
  \item Our $C^{k,\alpha}$ estimate implies a $C^{k-1,1}$ theory for
  $a$-harmonic functions on balls $B_R$, provided that we have already
  constructed an appropriate $k$th-order corrector on $B_R$. This is done in
  Lemma \ref{Ck-11Estimate}.
  \item At last, we are able to build our corrector, starting from small balls
  and iteratively doubling the size of our balls. At this point, we require the
  $C^{k-1,1}$ theory to show appropriate ($k$th-order) decay in the interior of
  the new contribution to the $k$th-order corrector entering at every scale.
  This iterative enlargement is carried out in Lemma
  \ref{IterativeConstructionkthCorrector}
  and finally enables us to prove Theorem \ref{ExistenceHigherOrderCorrector}.
  \item The $k$th-order Liouville principle stated in Corollary
  \ref{LiouvillePrincipleK} is an easy consequence of our $C^{k+1,\alpha}$
  large-scale excess-decay estimate.
\end{itemize}

\section{A $C^{2,\alpha}$ Large-Scale Regularity Theory for Homogeneous
Elliptic Equations with Random Coefficients}

For the reader's convenience, we shall first provide a proof for the
$C^{2,\alpha}$ case of our theorems, as in this case the proofs are less
technical while already containing the key ideas. In particular, the overall
structure of our proofs is the same as in the $C^{k,\alpha}$ case.
Since we shall use a somewhat simplified notation in the $C^{2,\alpha}$ case,
let us reformulate the $C^{2,\alpha}$ case of our theorems using this notation.

\begin{theorem}[Existence of second-order correctors]
\label{ExistenceSecondOrderCorrector}
Let $d\geq 2$ and suppose that the corrector $\phi$ and the flux-correction
potential $\sigma$ satisfy the growth assumption (\ref{Epsilon2Small}). Let
$r_0$ be large enough so that $\varepsilon_{2,r_0}\leq \varepsilon_0$ holds (the
existence of such $r_0$ is ensured by (\ref{Epsilon2Small})),
where $\varepsilon_0=\varepsilon_0(d,\lambda)>0$ is a constant defined in the
proof below.
Given any $E\in \mathbb{R}^{d\times d}$,
there exists a second-order corrector $\psi_E$ satisfying
\begin{align}
\label{EquationPsiE}
-\divv a\nabla \psi_E = E_{ij} \divv [\sigma_{ij}+\sigma_{ji}
+a(\phi_i e_j+\phi_j e_i)]
\end{align}
as well as
\begin{align*}
\sup_{R\geq r} \frac{1}{R} \left(\dashint_{B_R} |\nabla \psi_E|^2
~dx\right)^{1/2}
\leq C(d,\lambda) |E|
\varepsilon_{2,r}
\end{align*}
for any $r\geq r_0$.
Moreover, the corrector $\nabla \psi_E$ depends linearly on $E$.
\end{theorem}
Due to the linear dependence of $\psi_E$ on $E$, below we shall also write
$E_{ij}\psi_{ij}$ in place of $\psi_E$.

Note that our second-order correctors indeed enable us -- in conjunction with
the first-order correctors $\phi_i$ -- to correct $a_{hom}$-harmonic second-order
polynomials.
\begin{proposition}
\label{ExistenceCorrectorsPolynomial}
Let $d\geq 2$ and let $E\in \mathcal{E}$ (i.e. assume that the polynomial
$E_{ij} x_i x_j$ is $a_{hom}$-harmonic).
Suppose that $\psi_E$ satisfies \eqref{EquationPsiE}. We then have
\begin{align*}
-\divv a \nabla E_{ij} (x_i x_j+x_i \phi_j+\phi_i x_j+\psi_{ij})=0.
\end{align*}
\end{proposition}
Our $C^{2,\alpha}$ large-scale regularity theorem reads as follows.
\begin{theorem}[$C^{2,\alpha}$ large-scale excess-decay estimate]
\label{C2alphaExcessDecay}
Let $d\geq 2$ and suppose that (\ref{Epsilon2Small}) holds. Let $u$
be an $a$-harmonic function. Let $\psi_E$ be the second-order corrector
constructed in Theorem~\ref{ExistenceSecondOrderCorrector}.
Consider the second-order excess
\begin{align}
\label{SecondOrderExcessFirst}
\Exc&_2(r):=
\\ \nonumber
&\inf_{b\in \mathbb{R}^d, E\in \mathcal{E}}
\dashint_{B_r} \left|\nabla u-\nabla \big(b_i(x_i+\phi_i)
+E_{ij}(x_i x_j + x_i \phi_j + \phi_i x_j + \psi_{ij})\big)\right|^2 ~dx.
\end{align}
Let $0<\alpha<1$ and let
$r_0$ be large enough so that $\varepsilon_{2,r_0}\leq \varepsilon_0$ holds
(the existence of such $r_0$ is ensured by (\ref{Epsilon2Small})),
where $\varepsilon_0=\varepsilon_0(d,\lambda,\alpha)>0$ is a constant defined in
the proof below.
Then for all $r,R\geq r_0$ with $r<R$ the $C^{2,\alpha}$ excess-decay
estimate
\begin{align}
\label{SecondOrderExcessDecayFinal}
\Exc_2(r)\leq C(d,\lambda,\alpha)\left(\frac{r}{R}\right)^{2+2\alpha}
\Exc_2(R)
\end{align}
is satisfied.
\end{theorem}
Our large-scale excess-decay estimate entails the following $C^{2,\alpha}$
Liouville principle.
\begin{corollary}[$C^{2,\alpha}$ Liouville principle]
\label{LiouvillePrinciple}
Let $d\geq 2$ and suppose that the assumption (\ref{Epsilon2Small}) is
satisfied. Then the following property holds: Any $a$-harmonic function $u$
satisfying the growth condition
\begin{align*}
\liminf_{r\rightarrow \infty} \frac{1}{r^{2+\alpha}}
\left(\dashint_{B_r} \left|u\right|^2 ~dx\right)^{1/2}=0
\end{align*}
for some $\alpha\in (0,1)$ is of the form
\begin{align*}
u=a+b_i (x_i+\phi_i)+E_{ij}(x_i x_j +x_i \phi_j + \phi_i x_j + \psi_{ij})
\end{align*}
with some $a\in \mathbb{R}$, $b\in \mathbb{R}^d$, and $E\in
\mathcal{E}$ (i.e. some $E\in\mathbb{R}^{d\times
d}$ for which $E_{ij} x_i x_j$ is an $a_{hom}$-harmonic polynomial).
\end{corollary}

Let us start with the proof of Proposition
\ref{ExistenceCorrectorsPolynomial}, which only requires a simple computation.
\begin{proof}[Proof of Proposition \ref{ExistenceCorrectorsPolynomial}]
Making use of the fact that $E_{ij}((a_{hom})_{ij}+(a_{hom})_{ji})=0$ (in the
third step below), we compute
\begin{align*}
&E_{ij} \divv (\sigma_{ij}+\sigma_{ji})+E_{ij}\divv a
(\phi_i e_j+\phi_j e_i)
\\&
\stackrel{\eqref{EquationSigma}}{=}
E_{ij} q_{ij}+ E_{ij} q_{ji}
+E_{ij}\divv a (\phi_i e_j+\phi_j e_i)
\\&
\stackrel{\eqref{Equationq}}{=}
E_{ij} (a_{jk}((\Id)_{ik}+\partial_k \phi_i)-(a_{hom})_{ji})
+E_{ij} (a_{ik}((\Id)_{jk}+\partial_k \phi_j)-(a_{hom})_{ij})
\\&~~~
+E_{ij}\divv a (\phi_i e_j+\phi_j e_i)
\\&
=
E_{ij} (a_{jk}(\partial_k x_i +\partial_k \phi_i)
+a_{ik}(\partial_k x_j+\partial_k \phi_j))
\\&~~~
+E_{ij}\divv a (\phi_i \nabla x_j +\phi_j \nabla x_i)
\\&
=
E_{ij} (a\nabla (x_i +\phi_i) \cdot \nabla x_j
+a\nabla (x_j+\phi_j) \cdot \nabla x_i)
+E_{ij}\divv a (\phi_i \nabla x_j +\phi_j \nabla x_i)
\\&
\stackrel{\eqref{EquationCorrector}}{=}
E_{ij} \divv (x_j a\nabla (x_i +\phi_i)+x_i a\nabla (x_j+\phi_j))
+E_{ij}\divv a (\phi_i \nabla x_j +\phi_j \nabla x_i).
\end{align*}
We therefore obtain
\begin{align*}
&E_{ij} \divv (\sigma_{ij}+\sigma_{ji})+E_{ij}\divv a
(\phi_i e_j+e_i \phi_j)
\\&
=E_{ij} \divv a\nabla (x_i x_j + x_i \phi_j + \phi_i x_j),
\end{align*}
which together with \eqref{EquationPsiE} implies our proposition.
\end{proof}

\subsection{The $C^{2,\alpha}$ excess-decay estimate}

To establish our $C^{2,\alpha}$  excess-decay estimate, we make use of the
following lemma, which essentially generalizes Theorem~\ref{C2alphaExcessDecay}
to correctors which are only available on balls $B_R$.
\begin{lemma}
\label{C2alphaExcessDecayGeneral}
Let $d\geq 2$. For any $E\in
\mathcal{E}$, denote by $\tilde \psi_E$ a solution to the equation of
the second-order corrector \eqref{EquationPsiE} on the ball $B_R$ (without
boundary conditions); assume that $\tilde \psi_E$ depends linearly on $E$.
Set
\begin{align}
\label{DefinitionEpsPsiTilde}
\varepsilon_{\tilde\psi,r,R}:=\sup_{r\leq \rho\leq R} \rho^{-1}
\left(\max_{E\in \mathcal{E}, |E|=1} 
\dashint_{B_\rho} |\nabla \tilde \psi_E|^2 ~dx\right)^{1/2}.
\end{align}
For an $a$-harmonic function $u$ in $B_R$, consider the second-order excess
\begin{align}
\label{SecondOrderExcess}
&\widetilde\Exc_2(r):=
\\ \nonumber
&~~\inf_{b\in \mathbb{R}^d, E\in \mathcal{E}}
\dashint_{B_r} \left|\nabla u-\nabla \big(b_i(x_i+\phi_i)
+E_{ij}(x_i x_j + x_i \phi_j + \phi_i x_j +\tilde\psi_{ij})\big)\right|^2 ~dx.
\end{align}

For any $0<\alpha<1$ there exists a constant $\varepsilon_{min}>0$
depending only on $d$, $\lambda$, and $\alpha$ such that the following
assertion holds:

Suppose that $r_0>0$ satisfies
$\varepsilon_{r_0}+\varepsilon_{\tilde\psi,r_0,R}\leq \varepsilon_{min}$.
Then for all $r\in [r_0,R]$ the $C^{2,\alpha}$ excess-decay estimate
\begin{align}
\label{SecondOrderExcessDecay}
\widetilde\Exc_2(r)\leq C(d,\lambda,\alpha)\left(\frac{r}{R}\right)^{2+2\alpha}
\widetilde\Exc_2(R)
\end{align}
is satisfied.

Note that the infimum in \eqref{SecondOrderExcess} is actually attained, as the
average integral in the definition of $\widetilde\Exc_{2}(\rho)$ is a quadratic
functional of $b$ and $E$. Denote by $b^{\rho,min}$ and $E^{\rho,min}$ a
corresponding optimal choice of $b$ and $E$ in \eqref{SecondOrderExcess}. We
then have the estimates
\begin{align}
\label{EstimateEbDiff}
R^2|E^{r,min}-E^{R,min}|^2+|b^{r,min}-b^{R,min}|^2
\leq C(d,\lambda,\alpha) \widetilde\Exc_2(R)
\end{align}
and
\begin{align}
\label{EstimateEbSize}
R^2|E^{r,min}|^2+|b^{r,min}|^2
\leq C(d,\lambda,\alpha) \dashint_{B_R} |\nabla u|^2 ~dx.
\end{align}
\end{lemma}
\begin{proof}[Proof of Theorem \ref{C2alphaExcessDecay}]
Theorem \ref{C2alphaExcessDecay} obviously follows from Lemma
\ref{C2alphaExcessDecayGeneral} by setting $\tilde\psi_E:=\psi_E$, with $\psi_E$
being the second-order corrector whose existence is guaranteed by Theorem
\ref{ExistenceSecondOrderCorrector}.
\end{proof}

The following lemma is essentially a special case of our $C^{2,\alpha}$
large-scale excess-decay estimate Lemma \ref{C2alphaExcessDecayGeneral};
it entails the general case of Lemma \ref{C2alphaExcessDecayGeneral}, cf.
below.
\begin{lemma}
\label{ExcessDecayLemma}
Let $d\geq 2$ and let $R,r>0$ satisfy $r<R/4$ and $\varepsilon_R\leq 1$.
For any $E\in \mathcal{E}$,
denote by $\tilde \psi_E$ a solution to the equation of the second-order
corrector \eqref{EquationPsiE} on the ball $B_R$ (without boundary
conditions); assume that $\tilde \psi_E$ depends linearly on $E$.
For an $a$-harmonic function $u$ in $B_R$,
consider again the second-order excess \eqref{SecondOrderExcess}.
Then the excess on the smaller ball $B_r$ is estimated in terms of the excess
on the larger ball $B_R$ and our quantities $\varepsilon_R$ and
$\nabla \tilde \psi_E$:
We have
\begin{align*}
\widetilde\Exc_{2}(r)
&\leq C(d,\lambda)
\left[
\left(\frac{r}{R}\right)^{4}
+\left(\varepsilon_R^{2/(d+1)^2}+
R^{-2}\max_{E\in \mathcal{E},|E|=1}
\dashint_{B_R} |\nabla \tilde \psi_E|^2 ~dx
\right)\left(\frac{r}{R}\right)^{-d}
\right]
\\&~~~~~~\times
\widetilde\Exc_{2}(R).
\end{align*}
\end{lemma}

Before proving Lemma \ref{ExcessDecayLemma}, we would like to show how it
implies Lemma \ref{C2alphaExcessDecayGeneral}.
\begin{proof}[Proof of Lemma \ref{C2alphaExcessDecayGeneral}]
First choose $0<\theta\leq 1/4$ so small that the strict inequality
$C(d,\lambda)\theta^4<\theta^{2+2\alpha}$ is satisfied (with $C(d,\lambda)$
being the constant from Lemma \ref{ExcessDecayLemma}). Then, choose the
threshold $\varepsilon_{min}$ for
$\varepsilon_{r_0}+\varepsilon_{\tilde\psi,r_0,R}$ so small that the estimate
\begin{align*}
C(d,\lambda)
\left[\theta^4
+\left(\varepsilon_{r_0}^{2/(d+1)^2}
+\varepsilon_{\tilde\psi,r_0,R}^2
\right)\theta^{-d}
\right]
\leq \theta^{2+2\alpha}
\end{align*}
holds.

Let $M$ be the largest integer for which $\theta^M R \geq r$ holds.
Applying Lemma \ref{ExcessDecayLemma} inductively with $R_m:=\theta^{m-1} R$,
$r_m:=\theta^m R$ for $1\leq m\leq M$, we infer
\begin{align*}
\widetilde\Exc_2(\theta^M R)\leq (\theta^{2+2\alpha})^M \widetilde\Exc_2(R).
\end{align*}
Since we have trivially
\begin{align*}
\widetilde\Exc_2(r)\leq \left(\frac{r}{r_M}\right)^{-d} \widetilde\Exc_2(r_M)
\end{align*}
and since by definition of $M$ we have $r>\theta r_M$ and thus
$\theta^M<\theta^{-1}\frac{r}{R}$ (where we recall
$\theta=\theta(d,\lambda,\alpha)$), we infer
\begin{align*}
\widetilde\Exc_2(r)\leq C(d,\lambda,\alpha)\left(\frac{r}{R}\right)^{2+2\alpha}
\widetilde\Exc_2(R).
\end{align*}
It remains to show the estimates for $|b^{r,min}-b^{R,min}|$ and 
$|E^{r,min}-E^{R,min}|$ as well as the bounds for $|b^{r,min}|$ and
$|E^{r,min}|$. To do so, let us first estimate the differences
$|b^{R_m,min}-b^{r_m,min}|$ and $|E^{R_m,min}-E^{r_m,min}|$.
We have the estimate
\begin{align*}
&\dashint_{B_{r_m}} \big|\nabla (b_i^{R_m,min}-b_i^{r_m,min})
(x_i+\phi_i)
\\&~~~~~~~~~~~
+\nabla (E_{ij}^{R_m,min}-E_{ij}^{r_m,min}) (x_i x_j+x_i \phi_j+\phi_i x_j
+\tilde\psi_{ij})\big|^2
~dx
\\&
\leq
2\dashint_{B_{r_m}} \left|\nabla u-\nabla b_i^{r_m,min}(x_i+\phi_i)
-\nabla E_{ij}^{r_m,min}(x_i x_j + x_i \phi_j + \phi_i x_j + \tilde
\psi_{ij}) \right|^2 ~dx
\\&~~~
+2\dashint_{B_{r_m}} \left|\nabla u-\nabla b_i^{R_m,min}(x_i+\phi_i)
-\nabla E_{ij}^{R_m,min}(x_i x_j + x_i \phi_j + \phi_i x_j + \tilde \psi_{ij})
\right|^2 ~dx
\\&
\leq
2\widetilde\Exc_2(r_m)
+2\left(\frac{R_m}{r_m}\right)^d \widetilde\Exc_2(R_m)
\\&
\leq
C(d,\lambda,\alpha)
\left(\frac{r_m}{R}\right)^{2+2\alpha}
\widetilde\Exc_2(R)
+C(d,\lambda,\alpha)\theta^{-d}\left(\frac{R_m}{R}\right)^{2+2\alpha}
\widetilde\Exc_2(R)
\\&
\leq
C(d,\lambda,\alpha)
\left(\frac{r_m}{R}\right)^2
(\theta^{2\alpha})^m \widetilde\Exc_2(R).
\end{align*}
From Lemma \ref{QuantifiedLinearIndependenceLemma} below, we thus obtain
\begin{align*}
|b^{R_m,min}-b^{r_m,min}|
+R |E^{R_m,min}-E^{r_m,min}|
\leq
C(d,\lambda,\alpha) (\theta^{\alpha})^m \sqrt{\widetilde\Exc_2(R)}.
\end{align*}
Note that a similar estimate for the last increment $|b^{r_M,min}-b^{r,min}|+R
|E^{r_M,min}-E^{r,min}|$ can be derived analogously. Taking the sum with respect
to $m$ and recalling that $R_1=R$ and $r_m=R_{m+1}$, we finally deduce
\begin{align*}
|b^{R,min}-b^{r,min}|+R|E^{R,min}-E^{r,min}|
&\leq C(d,\lambda,\alpha) \sum_{m=0}^M (\theta^{\alpha})^m \sqrt{\widetilde
\Exc_2(R)}
\\
&\leq C(d,\lambda,\alpha) \sqrt{\widetilde
\Exc_2(R)}.
\end{align*}
It only
remains to establish the last estimate for $|b^{r,min}|$ and $|E^{r,min}|$.
By the previous estimate, it is sufficient to prove the corresponding
bound for $b^{R,min}$ and $E^{R,min}$. This in turn is a consequence of the
inequality
\begin{align*}
&\dashint_{B_R} \left|\nabla b_i^{R,min} (x_i+\phi_i)+\nabla E_{ij}^{R,min}
(x_i x_j+x_i \phi_j+\phi_i x_j+\tilde\psi_{ij})\right|^2 ~dx
\\&
\leq 2\widetilde\Exc_2(R)+2\dashint_{B_R} |\nabla u|^2 ~dx
\leq 4\dashint_{B_R} |\nabla u|^2 ~dx
\end{align*}
together with Lemma \ref{QuantifiedLinearIndependenceLemma} below.
\end{proof}
The following lemma quantifies the linear independence of the corrected
polynomials $x_i+\phi_i$, $E_{ij}(x_i x_j+x_i \phi_j+\phi_i
x_j+\tilde\psi_{ij})$; it is needed in the previous proof.
\begin{lemma}
\label{QuantifiedLinearIndependenceLemma}
Suppose that for every $E\in \mathcal{E}\setminus \{0\}$, the functions $\phi$
and $\tilde \psi_E$ satisfy
\begin{align*}
\rho^{-2}\dashint_{B_\rho} |\phi|^2 ~dx
+\rho^{-2}|E|^{-2}\dashint_{B_\rho} |\nabla \tilde \psi_E|^2 ~dx
\leq \varepsilon_0^2,
\end{align*}
where $\varepsilon_0=\varepsilon_0(d)$ is to be defined in the proof below.
Then for any $b\in \mathbb{R}^d$ and any $E\in \mathcal{E}$, we have the
estimate
\begin{align}
\nonumber
&|b|^2+\rho^2 |E|^2
\\&
\label{QuantifiedLinearIndependence}
\leq C(d)\dashint_{B_\rho} |\nabla b_i(x_i+\phi_i)+\nabla E_{ij} (x_i
x_j+x_i \phi_j+\phi_i x_j+\tilde\psi_{ij})|^2 ~dx.
\end{align}
\end{lemma}
\begin{proof}
Poincar\'e's inequality (with zero mean) and the triangle inequality imply
\begin{align*}
&\left(\dashint_{B_\rho} |\nabla b_i(x_i+\phi_i)+\nabla E_{ij} (x_i
x_j+x_i \phi_j+\phi_i x_j+\tilde\psi_{ij})|^2 ~dx\right)^{1/2}
\\
&\geq
\frac{1}{C(d)} \frac{1}{\rho}
\inf_{a\in \mathbb{R}}
\left(\dashint_{B_\rho} |b_i(x_i+\phi_i)
+E_{ij}(x_i x_j+x_i \phi_j+\phi_i x_j+\tilde\psi_{ij})-a|^2 ~dx\right)^{1/2}
\\
&\geq
\frac{1}{C(d)} \frac{1}{\rho}
\Bigg[
\inf_{a\in \mathbb{R}}
\left(\dashint_{B_\rho} |b_i x_i
+E_{ij}x_i x_j-a|^2 ~dx\right)^{1/2}
\\&~~~~~~~~~~~~~~
-\inf_{a\in \mathbb{R}} \left(\dashint_{B_\rho}
|b_i\phi_i+E_{ij}(x_i\phi_j+\phi_i x_j+\tilde\psi_{ij})-a|^2 ~dx\right)^{1/2}
\Bigg].
\end{align*}
On the one hand, by transversality of constant, linear, and quadratic
functions we have
\begin{align*}
&\frac{1}{\rho}\inf_{a\in \mathbb{R}}
\left(\dashint_{B_\rho} |b_i x_i
+E_{ij}x_i x_j-a|^2 ~dx\right)^{1/2}
\geq \frac{1}{C(d)}(|b|+\rho |E|).
\end{align*}
On the other hand, we have by the triangle inequality and Poincar\'e's
inequality
\begin{align*}
&\frac{1}{\rho}\inf_{a\in \mathbb{R}} \left(\dashint_{B_\rho}
|b_i\phi_i+E_{ij}(x_i\phi_j+\phi_i x_j+\tilde\psi_{ij})-a|^2 ~dx\right)^{1/2}
\\&
\leq C(d)
\left[(|b|+\rho|E|)\frac{1}{\rho}\left(\dashint_{B_\rho} |\phi|^2
~dx\right)^{1/2} +\rho |E| \frac{1}{\rho} \max_{\tilde E\in \mathcal{E},|\tilde E|=1}
\left(\dashint_{B_\rho} |\nabla \psi_{\tilde E}|^2 ~dx\right)^{1/2} \right].
\end{align*}
Putting these estimates together, by boundedness of the integrals in the
previous line by $\varepsilon_0^2 \rho^2$ our assertion is established.
\end{proof}

\begin{proof}[Proof of Lemma \ref{ExcessDecayLemma}]
In the proof of the lemma, we may assume that
\begin{align}
\label{AssumptionBEZero}
\widetilde\Exc_2(R)=\dashint_{B_R} |\nabla u|^2 ~dx.
\end{align}
To see this,
recall that the infimum in the definition of $\widetilde \Exc_2(R)$ is actually
attained. Denote the corresponding choices of $b$ and $E$ by $b^{min}$ and
$E^{min}$. Replacing $u$ by $u-b_i^{min}(x_i+\phi_i)-E_{ij}^{min}(x_i x_j+x_i
\phi_j+\phi_i x_j+\tilde\psi_{ij})$, we see that we
may indeed assume \eqref{AssumptionBEZero}:
The new function is also $a$-harmonic due to \eqref{EquationCorrector} and
Proposition \ref{ExistenceCorrectorsPolynomial}.

We then apply Lemma \ref{ApproximationConstantCoefficient} below to our
function $u$.
This yields an $a_{hom}$-harmonic function $u_{hom}$ close to $u$ which in
particular satisfies
\begin{align*}
\dashint_{B_{R/2}} |\nabla u_{hom}|^2 ~dx
\leq C(d,\lambda) \dashint_{B_R} |\nabla u|^2 ~dx.
\end{align*}
By inner regularity theory for elliptic equations with constant coefficients,
the $a_{hom}$-harmonic function $u_{hom}$ satisfies
\begin{align*}
&|\nabla u_{hom}(0)|
+R \sup_{B_{R/4}} |\nabla^2 u_{hom}|
+R^2 \sup_{B_{R/4}}|\nabla^3 u_{hom}|
\\&
\leq C(d,\lambda) \left(\dashint_{B_{R/2}}
|\nabla u_{hom}|^2 ~dx\right)^{1/2}
\leq C(d,\lambda) \left(\dashint_{B_{R}} |\nabla u|^2
~dx\right)^{1/2}.
\end{align*}
Let us define
\begin{align*}
b^{R,Taylor}:=&\nabla u_{hom}(0),
\\
E^{R,Taylor}:=&\nabla^2 u_{hom}(0).
\end{align*}
Since $-\divv a_{hom}\nabla u_{hom}=0$ holds, we infer
$E_{ij}^{R,Taylor}(a_{hom})_{ij}=0$ and therefore $E^{R,Taylor}\in \mathcal{E}$
(note that $E_{ij}^{R,Taylor}=E_{ji}^{R,Taylor}$).
By Taylor's expansion of $\nabla u_{hom}$ around $x=0$ we deduce for any $x\in
B_{R/4}$ the bound
\begin{align*}
\left|\nabla u_{hom}(x)
-b^{R,Taylor}-\frac{1}{2}E_{ij}^{R,Taylor}(x_j e_i+x_i e_j) \right|
\leq |x|^2 \sup_{B_{R/4}}|\nabla^3 u_{hom}|.
\end{align*}
Making use of the identity
\begin{align*}
&(\Id+(\nabla \phi)^t)\nabla u_{hom}
-\nabla \left(b_i^{R,Taylor} (x_i+\phi_i)
+\frac{1}{2}E_{ij}^{R,Taylor} (x_i x_j+x_i \phi_j+\phi_i x_j)\right)
\\&~~~
+\frac{1}{2}E_{ij}^{R,Taylor} (\phi_j e_i + \phi_i e_j)
\\&
=(\Id+(\nabla \phi)^t)
\left(\nabla u_{hom}(x)
-b^{R,Taylor}-\frac{1}{2}E_{ij}^{R,Taylor}(x_j e_i+x_i e_j)\right),
\end{align*}
the previous estimate yields in connection with the bound for $|\nabla^3
u_{hom}|$ and $r<R/4$
\begin{align*}
&\dashint_{B_r} \bigg|
(\Id+(\nabla \phi)^t)\nabla u_{hom}
-\nabla \left(b_i^{R,Taylor} (x_i+\phi_i)
+\frac{1}{2}E_{ij}^{R,Taylor} (x_i x_j+x_i \phi_j+\phi_i x_j)\right)
\\&~~~~~~~
+\frac{1}{2}E_{ij}^{R,Taylor} (\phi_j e_i + \phi_i e_j)\bigg|^2
~dx
\\&
\leq
C(d,\lambda) \left(\frac{r}{R}\right)^4
\dashint_{B_R} |\nabla u|^2 ~dx \times \dashint_{B_r} |\Id+(\nabla\phi)^t|^2
~dx.
\end{align*}
By the Caccioppoli inequality for the $a$-harmonic function $x_i+\phi_i$ (cf.
\eqref{EquationCorrector}), we have
\begin{align}
\label{CaccioppoliPhi}
\dashint_{B_r} |\Id+(\nabla \phi)^t|^2 ~dx
\leq \frac{C(d,\lambda)}{r^2} \dashint_{B_{2r}} |x+\phi|^2 ~dx
\leq C(d,\lambda) (1+\varepsilon_{2r}^2).
\end{align}
The approximation property of $u_{hom}+\phi_i \partial_i u_{hom}$ in $B_{R/2}$
from Lemma \ref{ApproximationConstantCoefficient} below implies
\begin{align*}
\dashint_{B_r} |\nabla u-\nabla (u_{hom}+\phi_i \partial_i u_{hom})|^2 ~dx
\leq
C(d,\lambda)\varepsilon_{R}^{2/(d+1)^2} \left(\frac{r}{R}\right)^{-d}
\dashint_{B_R} |\nabla u|^2 ~dx.
\end{align*}
Combining the last three estimates and the equality
\begin{align*}
&\nabla u
-\nabla \left(b_i^{R,Taylor} (x_i+\phi_i)
+\frac{1}{2}E_{ij}^{R,Taylor} (x_i x_j+x_i \phi_j+\phi_i x_j+\tilde\psi_{ij})\right)
\\
=&
\bigg[
(\Id+(\nabla \phi)^t)\nabla u_{hom}
-\nabla \left(b_i^{R,Taylor} (x_i+\phi_i)
+\frac{1}{2}E_{ij}^{R,Taylor} (x_i x_j+x_i \phi_j+\phi_i x_j)\right)
\\&~~~
+\frac{1}{2}E_{ij}^{R,Taylor}(\phi_j e_i + \phi_i e_j)\bigg]
-\frac{1}{2}E_{ij}^{R,Taylor} (\phi_j e_i+\phi_i e_j+\nabla\tilde\psi_{ij})
\\&
+\Big[\nabla u-\nabla (u_{hom}+\phi_i \partial_i u_{hom})\Big]
+\phi_i \nabla \partial_i u_{hom}
,
\end{align*}
we infer
\begin{align*}
&\dashint_{B_r} \bigg|
\nabla u
-\nabla \left(b_i^{R,Taylor} (x_i+\phi_i)
+\frac{1}{2}E_{ij}^{R,Taylor} (x_i x_j+x_i \phi_j+\phi_i x_j+\tilde\psi_{ij})\right)
\bigg|^2
~dx
\\&
\leq
4\dashint_{B_r} \bigg|
(\Id+(\nabla \phi)^t)\nabla u_{hom}
-\nabla \left(b_i^{R,Taylor} (x_i+\phi_i)
+\frac{1}{2}E_{ij}^{R,Taylor} (x_i x_j+x_i \phi_j+\phi_i x_j)\right)
\\&~~~~~~~~~~~~
+\frac{1}{2}E_{ij}^{R,Taylor}(\phi_j e_i + \phi_i e_j)\bigg|^2
~dx
\\&~~~
+4\dashint_{B_r} \bigg|\frac{1}{2}E_{ij}^{R,Taylor} (\phi_j e_i+\phi_i
e_j+\nabla\tilde\psi_{ij})\bigg|^2 ~dx
\\&~~~
+4\dashint_{B_r} | \nabla u-\nabla (u_{hom}+\phi_i \partial_i
u_{hom})|^2 ~dx
\\&~~~
+4\dashint_{B_r} |\phi_i \nabla \partial_i u_{hom}|^2 ~dx
\\&
\leq C(d,\lambda) 
\left(\frac{r}{R}\right)^4
\left(1+\varepsilon_r^2\right)
\dashint_{B_R} |\nabla u|^2 ~dx
\\&~~~
+C(d)|E^{R,Taylor}|^2
\left(r^2\varepsilon_r^2
+\max_{E\in \mathcal{E},|E|=1}
\dashint_{B_r} |\nabla \tilde\psi_E|^2 ~dx\right)
\\&~~~
+C(d,\lambda)\varepsilon_{R}^{2/(d+1)^2} \left(\frac{r}{R}\right)^{-d}
\dashint_{B_R} |\nabla u|^2 ~dx
\\&~~~
+C(d) r^2\varepsilon_r^2 \sup_{B_{R/4}} |\nabla^2 u_{hom}|^2.
\end{align*}
This finally yields in connection with the above bounds on $\nabla^2 u_{hom}$
in $B_{R/4}$ (recall that $E^{R,Taylor}=\nabla^2 u_{hom}(0)$)
\begin{align*}
&\dashint_{B_r} \bigg|
\nabla u
-\nabla \left(b_i^{R,Taylor} (x_i+\phi_i)
+\frac{1}{2}E_{ij}^{R,Taylor} (x_i x_j+x_i \phi_j+\phi_i x_j
+\tilde\psi_{ij})\right)
\bigg|^2 ~dx
\\
&\leq
C(d,\lambda) 
\left(\frac{r}{R}\right)^4
\left(1+\varepsilon_r^2\right)
\dashint_{B_R} |\nabla u|^2 ~dx
\\&~~~
+C(d,\lambda)R^{-2} \dashint_{B_R} |\nabla u|^2 ~dx
\left(r^2\varepsilon_r^2
+\max_{E\in \mathcal{E},|E|=1}
\dashint_{B_r} |\nabla \tilde\psi_E|^2 ~dx\right)
\\&~~~
+C(d,\lambda)\varepsilon_{R}^{2/(d+1)^2} \left(\frac{r}{R}\right)^{-d}
\dashint_{B_R} |\nabla u|^2 ~dx
\\&~~~
+C(d,\lambda) r^2\varepsilon_r^2 R^{-2} \dashint_{B_R} |\nabla u|^2 ~dx
\\&
\leq
C(d,\lambda)
\left[
\left(\frac{r}{R}\right)^4
+\left(\varepsilon_R^{2/(d+1)^2}
+R^{-2} \max_{E\in \mathcal{E},|E|=1} \dashint_{B_R} |\nabla \tilde\psi_E|^2 ~dx
\right) \left(\frac{r}{R}\right)^{-d}
\right]
\\&~~~~~~~
\times
\dashint_{B_R} |\nabla u|^2 ~dx,
\end{align*}
where in the last step we have used the inequality $\varepsilon_r^2\leq
\left(\frac{R}{r}\right)^d \varepsilon_R^2\leq \left(\frac{R}{r}\right)^d
\varepsilon_R^{2/(d+1)^2}$. The new bound directly implies the desired
estimate.
\end{proof}

\subsection{The $C^{1,1}$ excess-decay estimate}

We now show how our $C^{2,\alpha}$ excess-decay estimate for the second-order
excess $\widetilde\Exc_2$ from Lemma \ref{C2alphaExcessDecayGeneral}
entails a $C^{1,1}$ excess-decay estimate for the
first-order excess $\Exc$.
\begin{lemma}
\label{C11Estimate}
Let $d\geq 2$ and $R>0$. For any $E\in \mathcal{E}$, denote by
$\tilde \psi_E$ a solution to the equation of the second-order corrector
\eqref{EquationPsiE} on the ball $B_R$ (without boundary conditions); assume
that $\tilde \psi_E$ depends linearly on $E$.
There exists a constant $\varepsilon_{min}>0$ depending only on $d$ and
$\lambda$ such that the following assertion holds:

Suppose $r_0\in (0,R]$ is so large that $\varepsilon_{r_0}\leq
\varepsilon_{min}$ and
\begin{align*}
\sup_{r_0\leq \rho\leq R} \rho^{-1}
\left(\max_{E\in \mathcal{E},|E|=1}
\dashint_{B_\rho} |\nabla \tilde \psi_E|^2 ~dx\right)^{1/2}
\leq \varepsilon_{min}
\end{align*}
hold. Let $u$ be an $a$-harmonic function on $B_R$. Then there exists $b^R\in
\mathbb{R}^d$ for which the estimate
\begin{align*}
\dashint_{B_r} |\nabla u-\nabla b_i^R (x_i+\phi_i)|^2 ~dx
\leq C(d,\lambda)
\left(\frac{r}{R}\right)^2
\dashint_{B_R} |\nabla u|^2 ~dx
\end{align*}
holds for any $r\in [r_0,R]$.
Furthermore, $b^R$ depends linearly on $u$ and satisfies
\begin{align*}
|b^R|^2\leq C(d,\lambda) \dashint_{B_R} |\nabla u|^2 ~dx.
\end{align*}
\end{lemma}
\begin{proof}
In Lemma \ref{C2alphaExcessDecayGeneral}, fix $\alpha:=1/2$. We then easily
verify that Lemma \ref{C2alphaExcessDecayGeneral} is applicable in our
situation.
Set $b^R:=b^{r_0,min}$ and $E^R:=E^{r_0,min}$; this implies that $b^R$ depends
linearly on $u$.
The estimate \eqref{EstimateEbSize} takes the form
\begin{align*}
R^2|E^R|^2+|b^R|^2 \leq C(d,\lambda) \dashint_{B_R} |\nabla u|^2 ~dx.
\end{align*}
Furthermore, applying Lemma \ref{C2alphaExcessDecayGeneral} with $r_0$ playing
the role of $r$ and $r$ playing the role of $R$, we deduce from
\eqref{EstimateEbDiff}
\begin{align*}
r^2|E^R-E^{r,min}|^2+|b^R-b^{r,min}|^2 &\leq C(d,\lambda) \widetilde\Exc_2(r)
\\
\stackrel{\eqref{SecondOrderExcessDecay}}{\leq}
C(d,\lambda) \left(\frac{r}{R}\right)^{2+2\alpha}\widetilde\Exc_2(R)
&\leq C(d,\lambda) \left(\frac{r}{R}\right)^{2+2\alpha}
\dashint_{B_R} |\nabla u|^2 ~dx.
\end{align*}
We now estimate
\begin{align*}
&\dashint_{B_r} |\nabla u-\nabla b_i^R (x_i+\phi_i)|^2 ~dx
\\&
\leq
3\dashint_{B_r} \left|\nabla u-\nabla b_i^{r,min} (x_i+\phi_i)
-\nabla E_{ij}^{r,min}(x_i x_j + x_i \phi_j + \phi_i x_j +
\tilde\psi_{ij})\right|^2 ~dx
\\&~~~
+3\dashint_{B_r} \left|\nabla E_{ij}^{r,min}(x_i x_j + x_i \phi_j
+\phi_i x_j + \tilde\psi_{ij})\right|^2 ~dx
\\&~~~
+3\dashint_{B_r} |(b_i^{r,min}-b_i^R) \nabla (x_i+\phi_i)|^2 ~dx
\\&
\leq
3 \widetilde\Exc_2(r)
\\&~~~
+C(d) |E^{r,min}|^2 \left(\dashint_{B_r} |\phi|^2+ r^2 |\Id+(\nabla
\phi)^t|^2 ~dx
+\max_{E\in \mathcal{E}, |E|=1} \dashint_{B_r} |\nabla \tilde \psi_E|^2 ~dx
\right)
\\&~~~
+3 |b^{r,min}-b^R|^2 \dashint_{B_r} |\Id+(\nabla\phi)^t|^2 ~dx
\\&
\stackrel{(\ref{SecondOrderExcessDecay},\ref{CaccioppoliPhi})}{\leq}
C(d,\lambda) \left(\frac{r}{R}\right)^{2+2\alpha} \widetilde\Exc_2(R)
+C(d,\lambda) |E^{r,min}|^2 r^2 (\varepsilon_r^2 + (1+\varepsilon_{2r}^2) +
\varepsilon_{\tilde\psi,r_0,R}^2)
\\&~~~
+C(d,\lambda) |b^{r,min}-b^R|^2 (1+\varepsilon_{2r}^2)
\\&
\leq
C(d,\lambda) \left(\frac{r}{R}\right)^{2+2\alpha} \widetilde\Exc_2(R)
+C(d,\lambda) |E^{r,min}|^2 r^2
+C(d,\lambda) |b^{r,min}-b^R|^2.
\end{align*}
In conjunction with the two previous estimates, we infer
\begin{align*}
&\dashint_{B_r} |\nabla u-\nabla b_i^R (x_i+\phi_i)|^2 ~dx
\\&
\leq
C(d,\lambda) \left[\left(\frac{r}{R}\right)^{2+2\alpha}
+\left(\left(\frac{r}{R}\right)^2
+\left(\frac{r}{R}\right)^{2+2\alpha}\right)
+\left(\frac{r}{R}\right)^{2+2\alpha}
\right]
\dashint_{B_R} |\nabla u|^2~dx.
\end{align*}
Our lemma is therefore established.
\end{proof}

\subsection{Construction of second-order correctors}

Using the $C^{1,1}$ theory established in the previous subsection, we now
proceed to the construction of our second-order corrector. The following lemma
provides the inductive step; starting from a function which acts as a corrector
on a ball $B_R$, we construct a function acting as a corrector on the ball
$B_{2R}$.
\begin{lemma}
\label{IterativeConstructionSecondCorrector}
Let $d\geq 2$ and let $r_0>0$ satisfy the estimate $\varepsilon_{2,r_0}\leq
\varepsilon_0$, where $\varepsilon_0=\varepsilon_0(d,\lambda)$ is to be chosen
in the proof below.
Then the following implication holds:

Let $R=2^M r_0$ for some $M\in \mathbb{N}_0$.
Suppose that for every $E\in \mathbb{R}^{d\times d}$ we have a solution
$\psi^R_E$ to the equation
\begin{align*}
-\divv a\nabla \psi_E^R = E_{ij} \divv
\chi_{B_R}[\sigma_{ij}+\sigma_{ji}
+a(\phi_i e_j+\phi_j e_i)]
\end{align*}
subject to the growth condition
\begin{align*}
r^{-1}
\left(
\dashint_{B_r} |\nabla \psi_E^R|^2 ~dx
\right)^{1/2}
\leq C_1(d,\lambda) |E|
\sum_{m=0}^{M}
\min\{1,2^m r_0/r\}\varepsilon_{2^m r_0}
\end{align*}
for all $r\geq r_0$,
where $C_1(d,\lambda)$ is a sufficiently large constant to be chosen in the
proof below. Assume furthermore that $\psi^R_E$ depends linearly on $E$.

Then for every $E\in \mathbb{R}^{d\times d}$ there exists a solution
$\psi^{2R}_E$ to the equation
\begin{align*}
-\divv a\nabla \psi_E^{2R} = E_{ij} \divv
[\chi_{B_{2R}}(\sigma_{ij}+\sigma_{ji}+a (\phi_i e_j+\phi_j e_i))]
\end{align*}
subject to the growth condition
\begin{align*}
r^{-1}
\left(
\dashint_{B_r} |\nabla \psi_E^{2R}|^2 ~dx
\right)^{1/2}
\leq C_1(d,\lambda) |E|
\sum_{m=0}^{M+1}
\min\{1,2^m r_0/r\} \varepsilon_{2^m r_0}
\end{align*}
for all $r\geq r_0$.
Furthermore, $\psi_E^{2R}$ depends linearly on $E$ and we have
\begin{align*}
&r^{-1}
\left(\dashint_{B_r} |\nabla \psi_E^{2R}-\nabla \psi_E^R|^2
~dx\right)^{1/2}
\leq C_1(d,\lambda) |E|
\varepsilon_{2^{M+1} r_0}.
\end{align*}
\end{lemma}
\begin{proof}
To establish the lemma, we first note that the assumptions of the lemma ensure
that the $C^{1,1}$ excess-decay lemma (Lemma \ref{C11Estimate}) is applicable on
$B_R$ with $\tilde \psi_E:=\psi_E^{R}$. To see this, we estimate for any
$r\in [r_0,R]$
\begin{align*}
r^{-1}\left(
\dashint_{B_r} |\nabla \psi_E^R|^2 ~dx
\right)^{1/2}
\leq C_1(d,\lambda) |E|
\varepsilon_{2,r_0}
\leq C_1(d,\lambda) |E| \varepsilon_0.
\end{align*}
By choosing $\varepsilon_0>0$ small enough depending only on $d$ and $\lambda$
and $C_1$ (which is to be chosen at the end of this proof), we can ensure
that the assumption of Lemma~\ref{C11Estimate} regarding smallness of
$\varepsilon_{\tilde \psi,r_0,R}$ is satisfied.

Let now $\xi_E^R$ be the weak solution on $\mathbb{R}^d$ with square integrable
gradient, which is unique up to additive constants and whose existence follows
from the Lax-Milgram theorem, to the problem
\begin{align*}
-\divv a\nabla \xi_E^R = E_{ij} \divv
\chi_{B_{2R}-B_R}(\sigma_{ij}+\sigma_{ji})
+E_{ij}\divv \chi_{B_{2R}-B_R} a (\phi_i e_j+\phi_j e_i).
\end{align*}
Obviously, $\nabla \xi_E^R$ depends linearly on $E$; after fixing the additive
constant e.g. by requiring $\int_{B_1} \xi_E^R ~dx=0$, $\xi_E^R$ itself depends
linearly on $E$.
Furthermore, we have the bound
\begin{align*}
\int_{\mathbb{R}^d} |\nabla \xi_E^R|^2 ~dx
\leq C(\lambda) |E|^2 \int_{\mathbb{R}^d}
\chi_{B_{2R}-B_R}|\sigma|^2
+\chi_{B_{2R}-B_R} |\phi|^2 ~dx
\end{align*}
and therefore
\begin{align}
\label{EstimateXiRE}
\int_{\mathbb{R}^d} |\nabla \xi_E^R|^2 ~dx
\leq C(\lambda) |E|^2 R^{2+d} \varepsilon_{2R}^2.
\end{align}
As $\xi_E^R$ is $a$-harmonic in $B_R$, Lemma \ref{C11Estimate} now implies the
existence of some $b^R_E\in \mathbb{R}^d$ for which the estimates
\begin{align}
\label{EstimatebRE}
|b^R_E|^2\leq C(d,\lambda) \dashint_{B_R} |\nabla \xi_E^R|^2 ~dx
\leq C(d,\lambda) |E|^2 R^2 \varepsilon_{2R}^2
\end{align}
and
\begin{align*}
\dashint_{B_r} |\nabla \xi_E^R-\nabla (b_E^R)_i (x_i+\phi_i)|^2 ~dx
&\leq C(d,\lambda)
\left(\frac{r}{R}\right)^2
\dashint_{B_R} |\nabla \xi_E^R|^2 ~dx
\\&
\leq C(d,\lambda) |E|^2 r^2 \varepsilon_{2R}^2
\end{align*}
hold for all $r\in [r_0,R]$ and which linearly depends on $E$.

Furthermore, we have for $r>R$
\begin{align*}
&\dashint_{B_r} |\nabla \xi_E^R-\nabla (b_E^R)_i (x_i+\phi_i)|^2 ~dx
\\&
\stackrel{\eqref{CaccioppoliPhi}}{\leq}
C(d,\lambda) \left(r^{-d}\int_{B_r} |\nabla \xi_E^R |^2 ~dx
+|b_E^R|^2(1+\varepsilon_{2r}^2)\right)
\\&
\stackrel{(\ref{EstimateXiRE},\ref{EstimatebRE})}{\leq}
C(d,\lambda) |E|^2 R^2
\left(\left(\frac{R}{r}\right)^d+1+\varepsilon_{2r}^2
\right)\varepsilon_{2R}^2
\\&
\leq C(d,\lambda) |E|^2 R^2 \varepsilon_{2R}^2.
\end{align*}
The combination of both $r$-ranges yields
\begin{align}
\label{LastContributionC11}
\frac{1}{r}
\left(\dashint_{B_r} |\nabla \xi_E^R-\nabla (b_E^R)_i (x_i+\phi_i)|^2 ~dx
\right)^{1/2} \leq C(d,\lambda)  |E| \min\{1,2R/r\} \varepsilon_{2R}.
\end{align}
In total, we see that
\begin{align*}
\psi_E^{2R}:=\psi_E^R+\xi_E^R-(b_E^R)_i (x_i+\phi_i)
\end{align*}
is the desired function (note in particular that the last term is $a$-harmonic),
provided we choose $C_1$ to be the constant appearing in
\eqref{LastContributionC11}.
\end{proof}

We now establish existence of second-order correctors by means of the previous
lemma.
\begin{proof}[Proof of Theorem \ref{ExistenceSecondOrderCorrector}]
We just need to construct an ``initial'' second-order corrector
$\psi_E^{r_0}$ subject to the properties of Lemma
\ref{IterativeConstructionSecondCorrector}; then Lemma
\ref{IterativeConstructionSecondCorrector} yields a sequence
$(\psi_E^{2^m r_0})_m$ which is a Cauchy sequence in $H^1(B_R)$ for every $R>0$
due to the last estimate in the lemma and our assumption \eqref{Epsilon2Small}
which implies summability of $\varepsilon_{2^m r_0}$.
Thus, the limit $\psi_E$ satisfies the equation \eqref{EquationPsiE} in the
whole space, depends linearly on $E$, and satisfies the estimate
\begin{align*}
r^{-1}
\left(
\dashint_{B_r} |\nabla \psi_E|^2 ~dx
\right)^{1/2}
\leq C_1(d,\lambda) |E|
\sum_{m=0}^{\infty}
\min\{1,2^m r_0/r\} \varepsilon_{2^m r_0}
\end{align*}
for any $r\geq r_0$.

To construct $\psi_E^{r_0}$, just use Lax-Milgram to find the solution
$\psi_E^{r_0}$ on $\mathbb{R}^d$ with square-integrable gradient (unique up to
an additive constant) to the equation
\begin{align*}
-\divv a\nabla \psi_E^{r_0} = E_{ij} \divv[
\chi_{B_{r_0}}(\sigma_{ij}+\sigma_{ji}
+a (\phi_i e_j+\phi_j e_i))].
\end{align*}
Obviously, after fixing the additive constant appropriately $\psi_E^{r_0}$
depends linearly on $E$.
Furthermore, we have the energy estimate
\begin{align*}
\int_{\mathbb{R}^d} |\nabla \psi_E^{r_0}|^2 ~dx
\leq C(\lambda) |E|^2 \int_{\mathbb{R}^d}
|\chi_{B_{r_0}}\sigma|^2
+|\chi_{B_{r_0}} a \phi|^2 ~dx,
\end{align*}
i.e. for any $r\geq r_0$
\begin{align*}
\int_{B_r} |\nabla \psi_E^{r_0}|^2 ~dx
\leq C(d,\lambda) |E|^2 \int_{B_{r_0}} |\phi|^2+|\sigma|^2 ~dx
\end{align*}
and therefore
\begin{align*}
\dashint_{B_r} |\nabla \psi_E^{r_0}|^2 ~dx
&\leq C(d,\lambda) |E|^2 r^{-d} \varepsilon_{r_0}^2 r_0^{2+d}
\\&
\leq C(d,\lambda) |E|^2 r^2 \min\{1,(r_0/r)^2\}
\varepsilon_{r_0}^2.
\end{align*}
We note that this provides the starting point for Lemma
\ref{IterativeConstructionSecondCorrector}, possibly after enlarging the
constant $C_1$ in the statement thereof.
\end{proof}

\subsection{Proof of the $C^{2,\alpha}$ Liouville principle}

\label{SubsectionLiouville}

The $C^{2,\alpha}$ Liouville principle (Corollary \ref{LiouvillePrinciple}) is
an easy consequence of our large-scale excess-decay estimate
(Theorem~\ref{C2alphaExcessDecay}).
\begin{proof}[Proof of Corollary \ref{LiouvillePrinciple}]
Let $\alpha\in (0,1)$ be such that
\begin{align*}
\lim_{R\rightarrow \infty} \frac{1}{R^{2+\alpha}}
\left(\dashint_{B_R} \left|u\right|^2 ~dx\right)^{1/2}=0
\end{align*}
holds. By the Caccioppoli estimate, we deduce
\begin{align*}
\lim_{R\rightarrow \infty} \frac{1}{R^{1+\alpha}}
\left(\dashint_{B_R} \left|\nabla u\right|^2 ~dx\right)^{1/2}=0.
\end{align*}
Fix $r\geq r_0$. The excess-decay estimate from Theorem \ref{C2alphaExcessDecay}
yields together with the trivial bound $\Exc_2(R)\leq \dashint_{B_R} |\nabla
u|^2 ~dx$ that
\begin{align*}
\Exc_2(r)&\leq C(d,\lambda,\alpha)\left(\frac{r}{R}\right)^{2+2\alpha}
\Exc_2(R)
\\&
\leq C(d,\lambda,\alpha) r^{2+2\alpha} \left(\frac{1}{R^{1+\alpha}}
\left(\dashint_{B_R} |\nabla u|^2 ~dx\right)^{1/2}\right)^2.
\end{align*}
Passing to the limit $R\rightarrow \infty$, we deduce that
\begin{align*}
\Exc_2(r)=0
\end{align*}
holds for every $r\geq r_0$. Therefore, on every $B_r$ with $r\geq r_0$, $\nabla
u$ can be represented \emph{exactly} as the derivative of a corrected polynomial
of second order (since the infimum in the definition of $\Exc_2$ is actually
attained, as noted at the beginning of the proof of Lemma
\ref{ExcessDecayLemma}), i.e.
we have
\begin{align*}
\nabla u=\nabla b_i^r (x_i+\phi_i) + \nabla E_{ij}^r (x_i x_j + x_i \phi_j +
\phi_i x_j + \psi_{ij})
\end{align*}
in $B_r$ for some $b^r\in \mathbb{R}^d$ and some $E^r\in \mathcal{E}$.
It is not difficult to show that for $r$ large enough, the $b^r$ and $E^r$ are
actually independent of $r$ and define some common $b\in \mathbb{R}^d$ and
$E\in \mathcal{E}$: For example, one may use Lemma
\ref{C2alphaExcessDecayGeneral} to compare the $b^r$, $E^r$ for two different
radii $r_1,r_2\geq r_0$; the estimate for $|b^{r_1}-b^{r_2}|$ and
$|E^{r_1}-E^{r_2}|$ then contains the factor $\Exc_2(\max(r_1,r_2))$ and is
therefore zero.
Moreover, the gradient $\nabla u$ determines the function $u$ itself up to a
constant, i.e. we have
\begin{align*}
u=a + b_i (x_i+\phi_i) + E_{ij} (x_i x_j + x_i \phi_j + \phi_i x_j + \psi_{ij})
\end{align*}
for some $a\in \mathbb{R}$, some $b\in \mathbb{R}^d$, and some $E\in
\mathcal{E}\subset \mathbb{R}^{d\times d}$.
\end{proof}

\section{A $C^{k,\alpha}$ Large-Scale Regularity Theory for Elliptic Equations
with Random Coefficients}

We now generalize our proofs from the $C^{2,\alpha}$ case
in order to correct polynomials of order $k$ and
obtain our $C^{k,\alpha}$ large-scale regularity theory. We proceed by
induction in $k$.

In order to establish our $C^{k,\alpha}$ regularity theory, let  us first show
Proposition \ref{ExistenceCorrectorsPolynomialK}, which -- like the proof of
Proposition \ref{ExistenceCorrectorsPolynomial} in the $C^{2,\alpha}$ case --
only requires a simple computation.
\begin{proof}[Proof of Proposition \ref{ExistenceCorrectorsPolynomialK}]
Making use of the fact that we have $(a_{hom})_{ij} \partial_i \partial_j
P=0$ (in the third step below), we obtain
\begin{align*}
&-\nabla \cdot (\sigma_i \nabla \partial_i P)
\\&
=(\nabla \cdot \sigma_i) \cdot \nabla \partial_i P
\\&
\stackrel{\eqref{EquationSigma}}{=}
q_i \cdot \nabla \partial_i P
\\&
\stackrel{\eqref{Equationq}}{=}
a(e_i+\nabla \phi_i) \cdot \nabla \partial_i P
\\&
\stackrel{\eqref{EquationCorrector}}{=}
\nabla \cdot (\partial_i P\, a(e_i+\nabla \phi_i)).
\end{align*}
This yields
\begin{align*}
&\nabla \cdot ((\phi_i a-\sigma_i)\nabla \partial_i P)
\\&
=\nabla \cdot a(\phi_i \nabla \partial_i P+\partial_i P e_i+\partial_i P
\nabla \phi_i)
\\&
=\nabla \cdot a\nabla (P+\phi_i \partial_i P),
\end{align*}
which together with \eqref{EquationPsiP} implies our proposition.
\end{proof}

\subsection{The $C^{k,\alpha}$ excess-decay estimate}

To establish our $C^{k,\alpha}$  excess-decay estimate, we make use of the
following lemma, which essentially generalizes Theorem~\ref{CkalphaExcessDecay}
to correctors that are only available on balls $B_R$.
\begin{lemma}
\label{CkalphaExcessDecayGeneral}
Let $d\geq 2$ and $k\geq 2$. Suppose that Theorem
\ref{ExistenceHigherOrderCorrector} holds for orders $2$, \ldots, $k-1$, and
set $\psi_P\equiv 0$ for first-order polynomials $P$ to simplify notation.
For any $P\in \Pa{k}$, denote by $\tilde \psi_P$ a solution to the equation
\eqref{EquationPsiP} on the ball $B_R$ (without
boundary conditions); assume that the $\tilde \psi_P$ depend linearly on $P$.
Set
\begin{align}
\label{DefinitionEpsPsiTildePk}
\varepsilon_{\tilde\psi,r,R}:=\sup_{r\leq \rho\leq R} \rho^{-(k-1)}
\left(\max_{P\in \Pa{k}, ||P||=1} 
\dashint_{B_\rho} |\nabla \tilde \psi_{P}|^2 ~dx\right)^{1/2}.
\end{align}
For an $a$-harmonic function $u$ in $B_R$, consider the $k$th-order excess
\begin{align}
\label{kthOrderExcess}
&\widetilde\Exc_k(r):=
\\ \nonumber
&~~\inf_{P_\kappa\in \Pa{\kappa}}
\dashint_{B_r} \bigg|\nabla u-\nabla \big(\sum_{\kappa=1}^{k-1}
(P_\kappa+\phi_i\partial_i P_\kappa+\psi_{P_\kappa})
+(P_k+\phi_i\partial_i P_k+\tilde\psi_{P_k})
\big)\bigg|^2 ~dx.
\end{align}

For any $0<\alpha<1$ there exists a constant $\varepsilon_{min}>0$
depending only on $d$, $k$, $\lambda$, and $\alpha$ such that the following
assertion holds:

Suppose that $r_0>0$ satisfies
$\varepsilon_{2,r_0}+\varepsilon_{\tilde\psi,r_0,R}\leq \varepsilon_{min}$.
Then for all $r\in [r_0,R]$ the $C^{k,\alpha}$ excess-decay estimate
\begin{align}
\label{kthOrderExcessDecay}
\widetilde\Exc_k(r)\leq
C(d,k,\lambda,\alpha)\left(\frac{r}{R}\right)^{2(k-1)+2\alpha}
\widetilde\Exc_k(R)
\end{align}
is satisfied.

Note that the infimum in \eqref{kthOrderExcess} is actually attained, as the
average integral in the definition of $\widetilde\Exc_{2}(\rho)$ is a quadratic
functional of
$P_\kappa$. Denote by
$P_\kappa^{\rho,min}$ a corresponding optimal choice of
$P_\kappa$ in \eqref{kthOrderExcess}. We then have the estimates
\begin{align}
\label{EstimatePbDiff}
\sum_{\kappa=1}^k R^{2(\kappa-1)} ||P_\kappa^{r,min}-P_\kappa^{R,min}||^2
\leq C(d,k,\lambda,\alpha) \widetilde\Exc_k(R)
\end{align}
and
\begin{align}
\label{EstimatePbSize}
\sum_{\kappa=1}^k R^{2(\kappa-1)} ||P_\kappa^{r,min}||^2
\leq C(d,k,\lambda,\alpha) \dashint_{B_R} |\nabla u|^2 ~dx.
\end{align}
\end{lemma}
\begin{proof}[Proof of Theorem \ref{CkalphaExcessDecay}]
Once we have shown Theorem \ref{ExistenceHigherOrderCorrector},
Theorem \ref{CkalphaExcessDecay} obviously follows from Lemma
\ref{CkalphaExcessDecayGeneral} by setting $\tilde\psi_{P_k}:=\psi_{P_k}$, with
$\psi_{P_k}$ being the $k$th-order corrector whose existence is established in
Theorem \ref{ExistenceHigherOrderCorrector}.
\end{proof}
The following lemma is essentially a special case of our $C^{k,\alpha}$
large-scale excess-decay estimate Lemma \ref{CkalphaExcessDecayGeneral};
it entails the general case of Lemma \ref{CkalphaExcessDecayGeneral}, cf.
below.
\begin{lemma}
\label{ExcessDecayLemmak}
Let $d\geq 2$, $k\geq 2$, and let $R,r>0$ satisfy $r<R/4$ and
$\varepsilon_{2,R}\leq \varepsilon_0(d,k-1,\lambda)$, with
$\varepsilon_0(d,k-1,\lambda)$ being the constant from Theorem
\ref{ExistenceHigherOrderCorrector} for the orders $2$, \ldots, $k-1$.
Assume that Theorem
\ref{ExistenceHigherOrderCorrector} holds for orders $2$, \ldots, $k-1$, and
let $\psi_P\equiv 0$ for linear polynomials $P$ in order to simplify notation.
For any $P\in \Pa{\kappa}$, denote by
$\tilde \psi_P$ a solution to the equation
\eqref{EquationPsiP} on the ball $B_R$ (without boundary conditions); assume
that $\tilde \psi_P$ depends linearly on $P$.
For an $a$-harmonic function $u$ on $B_R$, consider again again the $k$th-order
excess \eqref{kthOrderExcess}.
Then the excess on the smaller ball $B_r$ is estimated in terms of the excess
on the larger ball $B_R$ and our quantities $\varepsilon_{2,R}$ and $\nabla
\tilde \psi_P$:
We have
\begin{align*}
&\widetilde\Exc_k(r)
\leq
C(d,k,\lambda)
\widetilde\Exc_k(R)
\\&~~~~\times
\Bigg[
\left(\frac{r}{R}\right)^{2k}
+\left(\varepsilon_{2,R}^{2/(d+1)^2}
+R^{-2(k-1)}\max_{P\in \Pa{k},||P||=1} \dashint_{B_R} |\nabla \tilde\psi_P|^2
~dx \right)\left(\frac{r}{R}\right)^{-d}
\Bigg].
\end{align*}
\end{lemma}
Before proving Lemma \ref{ExcessDecayLemmak}, we would like to show how it
implies Lemma \ref{CkalphaExcessDecayGeneral}.
\begin{proof}[Proof of Lemma \ref{CkalphaExcessDecayGeneral}]
First choose $0<\theta\leq 1/4$ so small that the strict inequality
$C(d,k,\lambda)\theta^{2k}<\theta^{2(k-1)+2\alpha}$ is satisfied (with
$C(d,k,\lambda)$ being the constant from Lemma \ref{ExcessDecayLemmak}). Then,
choose the threshold $\varepsilon_{min}$ for
$\varepsilon_{2,r_0}+\varepsilon_{\tilde\psi,r_0,R}$ so small that the estimate
\begin{align*}
C(d,k,\lambda)
\left[\theta^{2k}
+\left(\varepsilon_{2,r_0}^{2/(d+1)^2}
+\varepsilon_{\tilde\psi,r_0,R}^2
\right)\theta^{-d}
\right]
\leq \theta^{2(k-1)+2\alpha}
\end{align*}
holds.

Let $M$ be the largest integer for which $\theta^M R \geq r$ holds.
Applying Lemma \ref{ExcessDecayLemmak} inductively with $R_m:=\theta^{m-1} R$,
$r_m:=\theta^m R$ for $1\leq m\leq M$, we infer
\begin{align*}
\widetilde\Exc_k(\theta^M R)\leq (\theta^{2(k-1)+2\alpha})^M
\widetilde\Exc_k(R).
\end{align*}
Since we have trivially
\begin{align*}
\widetilde\Exc_k(r)\leq
\left(\frac{r}{r_M}\right)^{-d} \widetilde\Exc_k(r_M)
\end{align*}
and since by definition of $M$ we have $r>\theta r_M$ and thus
$\theta^M<\theta^{-1}\frac{r}{R}$ (where we recall
$\theta=\theta(d,k,\lambda,\alpha)$), we infer
\begin{align*}
\widetilde\Exc_k(r)\leq
C(d,k,\lambda,\alpha)\left(\frac{r}{R}\right)^{2(k-1)+2\alpha}
\widetilde\Exc_k(R).
\end{align*}
It remains to show the estimates for
$||P_\kappa^{r,min}-P_\kappa^{R,min}||$ as well as the bounds for
$||P_\kappa^{r,min}||$. To do so, let us first estimate the differences
$||P_\kappa^{R_m,min}-P_\kappa^{r_m,min}||$ of
two successive polynomials.
We have the estimate
\begin{align*}
&\dashint_{B_{r_m}} \bigg|
\nabla \sum_{\kappa=1}^{k-1}
\Big(P_\kappa^{R_m,min}-P_\kappa^{r_m,min}+\phi_i\partial_i
(P_\kappa^{R_m,min}-P_\kappa^{r_m,min})
+\psi_{P_\kappa^{R_m,min}-P_\kappa^{r_m,min}}\Big)
\\&~~~~~~~~
+\nabla \Big(P_k^{R_m,min}-P_k^{r_m,min}+\phi_i\partial_i
(P_k^{R_m,min}-P_k^{r_m,min})
+\tilde\psi_{P_k^{R_m,min}-P_k^{r_m,min}}\Big)
\bigg|^2
~dx
\\&
\leq
2\dashint_{B_{r_m}} \bigg|\nabla u
-\nabla \sum_{\kappa=1}^{k-1}
\Big(P_\kappa^{r_m,min}+\phi_i\partial_i P_\kappa^{r_m,min}
+\psi_{P_\kappa^{r_m,min}}\Big)
\\&~~~~~~~~~~~~~~~~~~~~~~~
-\nabla \Big(P_k^{r_m,min}+\phi_i\partial_i P_k^{r_m,min}
+\tilde\psi_{P_k^{r_m,min}}\Big)
\bigg|^2 ~dx
\\&~~~
+2\dashint_{B_{r_m}} \bigg|\nabla u
-\nabla \sum_{\kappa=1}^{k-1}
\Big(P_\kappa^{R_m,min}+\phi_i\partial_i P_\kappa^{R_m,min}
+\psi_{P_\kappa^{R_m,min}}\Big)
\\&~~~~~~~~~~~~~~~~~~~~~~~~~~~
-\nabla \Big(P_k^{R_m,min}+\phi_i\partial_i P_k^{R_m,min}
+\tilde\psi_{P_k^{R_m,min}}\Big)
\bigg|^2 ~dx
\\&
\leq
2\widetilde\Exc_k(r_m)
+2\left(\frac{R_m}{r_m}\right)^d \widetilde\Exc_k(R_m)
\\&
\leq
C(d,k,\lambda,\alpha)
\left(\frac{r_m}{R}\right)^{2(k-1)+2\alpha}
\widetilde\Exc_k(R)
+C(d,k,\lambda,\alpha)\theta^{-d}\left(\frac{R_m}{R}\right)^{2(k-1)+2\alpha}
\widetilde\Exc_k(R)
\\&
\leq
C(d,k,\lambda,\alpha)
\left(\frac{r_m}{R}\right)^{2(k-1)}
(\theta^{2\alpha})^m \widetilde\Exc_k(R).
\end{align*}
From Lemma \ref{QuantifiedLinearIndependenceLemmaPk} below, we thus obtain
\begin{align*}
&
\sum_{\kappa=1}^k R^{\kappa-1} ||P_\kappa^{R_m,min}-P_\kappa^{r_m,min}||
\leq
C(d,k,\lambda,\alpha) (\theta^{\alpha})^m \sqrt{\widetilde\Exc_k(R)}.
\end{align*}
A similar estimate for the last increment
$\sum_{\kappa=1}^k R^{\kappa-1} ||P_\kappa^{r_M,min}-P_\kappa^{r,min}||$ can be
derived analogously. Taking the sum with respect to $m$ and recalling that
$R_1=R$ and $r_m=R_{m+1}$, we finally deduce
\begin{align*}
&
\sum_{\kappa=1}^k R^{\kappa-1} ||P_\kappa^{R,min}-P_\kappa^{r,min}||
\\
&\leq C(d,k,\lambda,\alpha) \sum_{m=1}^M (\theta^{\alpha})^m \sqrt{\widetilde
\Exc_k(R)}
\\
&\leq C(d,k,\lambda,\alpha) \sqrt{\widetilde
\Exc_k(R)}.
\end{align*}
It only
remains to establish the last estimate for
$||P_\kappa^{r,min}||$.
By the previous estimate, it is sufficient to prove the corresponding
bound for
$||P_\kappa^{R,min}||$. This in turn is a consequence of the obvious inequality
\begin{align*}
&\dashint_{B_R} \bigg|
\nabla \sum_{\kappa=1}^{k-1}
\Big(P_\kappa^{R,min}+\phi_i\partial_i P_\kappa^{R,min}
+\psi_{P_\kappa^{R,min}}\Big)
\\&~~~~~~~~~~
+\nabla \Big(P_k^{R,min}+\phi_i\partial_i P_k^{R,min}
+\tilde\psi_{P_k^{R,min}}\Big)
\bigg|^2 ~dx
\\&
\leq 2\widetilde\Exc_k(R)+2\dashint_{B_R} |\nabla u|^2 ~dx
\leq 4\dashint_{B_R} |\nabla u|^2 ~dx
\end{align*}
in conjunction with Lemma \ref{QuantifiedLinearIndependenceLemmaPk} below.
\end{proof}

The following lemma quantifies the linear independence of the corrected
polynomials
$P_\kappa+\phi_i \partial_i P_\kappa+\psi_{P_\kappa}$ (with $1\leq \kappa\leq
k$); it is needed for the previous proof.
\begin{lemma}
\label{QuantifiedLinearIndependenceLemmaPk}
Suppose that
the functions $\phi$ and $\tilde \psi_{P_\kappa}$ ($2\leq
\kappa\leq k$) satisfy
\begin{align*}
\rho^{-2}\dashint_{B_\rho} |\phi|^2 ~dx
+\sum_{\kappa=2}^k
\rho^{-2(\kappa-1)}\max_{P\in \Pa{\kappa}, ||P||=1} ||P||^{-2}
\dashint_{B_\rho} |\nabla \tilde \psi_{P}|^2
~dx \leq \varepsilon_0^2,
\end{align*}
where $\varepsilon_0=\varepsilon_0(d,k)$ is to be defined in the proof below.
Set $\tilde \psi_{P}\equiv 0$ for linear polynomials $P$ in order to simplify
notation.
Then for
any $P_\kappa\in \Pa{\kappa}$ ($1\leq
\kappa\leq k$) we have the estimate
\begin{align}
\sum_{\kappa=1}^{k}\rho^{2(\kappa-1)}||P_\kappa||^2
\label{QuantifiedLinearIndependencePk}
\leq C(d,k)
\dashint_{B_\rho} \bigg|
\nabla \sum_{\kappa=1}^{k}(P_\kappa+\phi_i
\partial_i P_\kappa+\tilde\psi_{P_\kappa})\bigg|^2 ~dx.
\end{align}
\end{lemma}
\begin{proof}
Poincar\'e's inequality (with zero mean) and the triangle inequality imply
\begin{align*}
&\left(\dashint_{B_\rho} \bigg|
\nabla \sum_{\kappa=1}^{k}(P_\kappa+\phi_i
\partial_i P_\kappa+\tilde\psi_{P_\kappa})\bigg|^2 ~dx\right)^{1/2}
\\
&\geq
\frac{1}{C(d)} \frac{1}{\rho}
\inf_{a\in \mathbb{R}}
\left(\dashint_{B_\rho} \bigg|
\sum_{\kappa=1}^{k}(P_\kappa+\phi_i \partial_i
P_\kappa+\tilde\psi_{P_\kappa}) -a
\bigg|^2 ~dx\right)^{1/2}
\\
&\geq
\frac{1}{C(d)} \frac{1}{\rho}
\Bigg[
\inf_{a\in \mathbb{R}}
\left(\dashint_{B_\rho} \bigg|
\sum_{\kappa=1}^{k}P_\kappa
-a\bigg|^2 ~dx\right)^{1/2}
\\&~~~~~~~~~~~~~~
-\inf_{a\in \mathbb{R}} \left(\dashint_{B_\rho}
\bigg|
\sum_{\kappa=1}^{k}(\phi_i \partial_i P_\kappa+\tilde\psi_{P_\kappa})
-a\bigg|^2 ~dx\right)^{1/2}
\Bigg].
\end{align*}
On the one hand, by transversality of constant, linear, homogeneous
second-order, \ldots, and homogeneous $k$th-order polynomials we have
\begin{align*}
&\frac{1}{\rho}\inf_{a\in \mathbb{R}}
\left(\dashint_{B_\rho} \bigg|
\sum_{\kappa=1}^{k}P_\kappa
-a\bigg|^2 ~dx\right)^{1/2}
\geq \frac{1}{C(d,k)}
\sum_{\kappa=1}^{k}\rho^{\kappa-1} ||P_\kappa||
.
\end{align*}
On the other hand, we have by the triangle inequality and Poincar\'e's
inequality
\begin{align*}
&\frac{1}{\rho}\inf_{a\in \mathbb{R}} \left(\dashint_{B_\rho}
\bigg|
\sum_{\kappa=1}^{k}(\phi_i \partial_i
P_\kappa+\tilde\psi_{P_\kappa})
-a\bigg|^2 ~dx\right)^{1/2}
\\&
\leq C(d,k)
\Bigg[\bigg(
\sum_{\kappa=1}^k
\rho^{\kappa-1}||P_\kappa||\bigg)
\frac{1}{\rho}\left(\dashint_{B_\rho} |\phi|^2 ~dx\right)^{1/2}
\\&~~~~~~~~~~~~~~~
+\sum_{\kappa=2}^k
\rho^{\kappa-1}||P_\kappa||
\frac{1}{\rho^{\kappa-1}}
\max_{P\in \P{\kappa},||P||=1}
\left(\dashint_{B_\rho} |\nabla \tilde\psi_{P}|^2 ~dx\right)^{1/2}
\Bigg].
\end{align*}
Putting these estimates together, by boundedness of the integrals in the
previous line by $\varepsilon_0^2 \rho^{2(\kappa-1)}$ our assertion is
established.
\end{proof}

\begin{proof}[Proof of Lemma \ref{ExcessDecayLemmak}]
In the proof of the lemma, we may assume that
\begin{align}
\label{AssumptionBPZero}
\widetilde\Exc_k(R)=\dashint_{B_R} |\nabla u|^2 ~dx.
\end{align}
To see this,
recall that the infimum in the definition of $\widetilde \Exc_k(R)$ is actually
attained. Denote the corresponding choices of
$P_\kappa$ by
$P_\kappa^{min}$. Replacing $u$ by
$u
-\sum_{\kappa=1}^{k-1}(P_\kappa^{min} + \phi_i
\partial_i P_\kappa^{min} + \psi_{P_\kappa^{min}})
-(P_k^{min} + \phi_i \partial_i P_k^{min} + \tilde\psi_{P_k^{min}}))$, we see
that we may indeed assume \eqref{AssumptionBPZero}:
The new function is also $a$-harmonic due to \eqref{EquationCorrector} and
Proposition \ref{ExistenceCorrectorsPolynomialK}.

We then apply Lemma \ref{ApproximationConstantCoefficient} below to our
function $u$.
This yields an $a_{hom}$-harmonic function $u_{hom}$ close to $u$ which in
particular satisfies
\begin{align*}
\dashint_{B_{R/2}} |\nabla u_{hom}|^2 ~dx
\leq C(d,\lambda) \dashint_{B_R} |\nabla u|^2 ~dx.
\end{align*}
By inner regularity theory for elliptic equations with constant coefficients,
the $a_{hom}$-harmonic function $u_{hom}$ satisfies
\begin{align}
\label{EstimateDuhom}
&|\nabla u_{hom}(0)|
+R \sup_{B_{R/4}} |\nabla^2 u_{hom}|
+\sum_{\kappa=2}^k R^{\kappa} \sup_{B_{R/4}}|\nabla^{\kappa+1} u_{hom}|
\\&
\nonumber
\leq C(d,k,\lambda) \left(\dashint_{B_{R/2}}
|\nabla u_{hom}|^2 ~dx\right)^{1/2}
\leq C(d,k,\lambda) \left(\dashint_{B_{R}} |\nabla u|^2
~dx\right)^{1/2}.
\end{align}
Let
$P_\kappa^{R,Taylor}$ (for $1\leq \kappa\leq k$) be the term of order 
$\kappa$ in the Taylor expansion of $u_{hom}$ at $x_0=0$.
We now show (for $\kappa\geq 2$, as for $\kappa=1$ this assertion is trivial)
that $P_\kappa^{R,Taylor}\in \Pa{\kappa}$.
The term-wise Hessian of the Taylor series of
$u_{hom}$ yields the Taylor series of $\nabla^2 u_{hom}$. We now know that
$a_{hom}:\nabla^2 u_{hom}=0$; thus, the Taylor series of $a_{hom}:\nabla^2
u_{hom}$ is identically zero and
by equating the coefficients we deduce $a_{hom}:\nabla^2
P_\kappa^{R,Taylor}=0$ for $2\leq \kappa\leq k$.

As the term-wise derivative of the Taylor series of $u_{hom}$ yields the
Taylor series of $\nabla u_{hom}$, we obtain by the standard error estimate
for the Taylor expansion of $\nabla u_{hom}$ at $x_0=0$ for any $x\in B_{R/4}$
the estimate
\begin{align*}
\bigg|\nabla u_{hom}(x)
-\sum_{\kappa=1}^k \nabla P_\kappa^{R,Taylor}(x)
\bigg|
\leq |x|^k \sup_{B_{R/4}}|\nabla^{k+1} u_{hom}|.
\end{align*}

Making use of the identity
\begin{align*}
&(\Id+(\nabla \phi)^t)\nabla u_{hom}
-\nabla \sum_{\kappa=1}^{k} (P_\kappa^{R,Taylor}+\phi_i\partial_i
P_\kappa^{R,Taylor})
\\&
+\sum_{\kappa=2}^{k} \phi_i \nabla \partial_i P_\kappa^{R,Taylor}
\\&
=(\Id+(\nabla \phi)^t)
\left(\nabla u_{hom}(x)
-\sum_{\kappa=1}^k \nabla P_\kappa^{R,Taylor}(x)\right),
\end{align*}
the previous estimate yields in connection with the bound for $|\nabla^{k+1}
u_{hom}|$ and $r<R/4$
\begin{align*}
&\dashint_{B_r} \bigg|
(\Id+(\nabla \phi)^t)\nabla u_{hom}
-\nabla \sum_{\kappa=1}^{k} (P_\kappa^{R,Taylor}+\phi_i\partial_i
P_\kappa^{R,Taylor})
\\&~~~~~~
+\sum_{\kappa=2}^{k} \phi_i \nabla \partial_i P_\kappa^{R,Taylor}
\bigg|^2
~dx
\\&
\leq
C(d,k,\lambda) \left(\frac{r}{R}\right)^{2k}
\dashint_{B_R} |\nabla u|^2 ~dx \times \dashint_{B_r} |\Id+(\nabla\phi)^t|^2
~dx.
\end{align*}
By the Caccioppoli inequality for the $a$-harmonic function $x_i+\phi_i$ (cf.
\eqref{EquationCorrector}), we have
\begin{align}
\label{CaccioppoliPhi2}
\dashint_{B_r} |\Id+(\nabla \phi)^t|^2 ~dx
\leq \frac{C(d,\lambda)}{r^2} \dashint_{B_{2r}} |x+\phi|^2 ~dx
\leq C(d,\lambda) (1+\varepsilon_{2r}^2).
\end{align}
The approximation property of $u_{hom}+\phi_i \partial_i u_{hom}$ in $B_{R/2}$
from Lemma \ref{ApproximationConstantCoefficient} below implies
\begin{align*}
\dashint_{B_r} |\nabla u-\nabla (u_{hom}+\phi_i \partial_i u_{hom})|^2 ~dx
\leq
C(d,\lambda)\varepsilon_{R}^{2/(d+1)^2} \left(\frac{r}{R}\right)^{-d}
\dashint_{B_R} |\nabla u|^2 ~dx.
\end{align*}
Combining the last three estimates and the equality
\begin{align*}
&\nabla u
-\nabla \sum_{\kappa=1}^{k-1}
\Big(P_\kappa^{R,Taylor}
+\phi_i \partial_i P_\kappa^{R,Taylor}
+\psi_{P_\kappa^{R,Taylor}}
\Big)
\\&
-\nabla \Big(P_k^{R,Taylor}
+\phi_i \partial_i P_k^{R,Taylor}
+\tilde \psi_{P_k^{R,Taylor}}
\Big)
\\
=&
\bigg[
(\Id+(\nabla \phi)^t)\nabla u_{hom}
-\nabla \sum_{\kappa=1}^{k}
\Big(P_\kappa^{R,Taylor}
+\phi_i \partial_i P_\kappa^{R,Taylor}
\Big)
\\&~~~
+\sum_{\kappa=2}^k \phi_i \nabla \partial_i P_\kappa^{R,Taylor}\bigg]
-\sum_{\kappa=2}^k \phi_i \nabla \partial_i P_\kappa^{R,Taylor}
-\sum_{\kappa=2}^{k-1} \nabla \psi_{P_\kappa^{R,Taylor}}
-\nabla \tilde\psi_{P_k^{R,Taylor}}
\\&
+\Big[\nabla u-\nabla (u_{hom}+\phi_i \partial_i u_{hom})\Big]
+\phi_i \nabla \partial_i u_{hom}
,
\end{align*}
we infer
\begin{align*}
&\dashint_{B_r} \bigg|
\nabla u
-\nabla \sum_{\kappa=1}^{k-1}
\Big(P_\kappa^{R,Taylor}
+\phi_i \partial_i P_\kappa^{R,Taylor}
+\psi_{P_\kappa^{R,Taylor}}
\Big)
\\&~~~~~~
-\nabla \Big(P_k^{R,Taylor}
+\phi_i \partial_i P_k^{R,Taylor}
+\tilde \psi_{P_k^{R,Taylor}}
\Big)
\bigg|^2
~dx
\\&
\leq
6\dashint_{B_r} \bigg|
(\Id+(\nabla \phi)^t)\nabla u_{hom}
-\nabla \sum_{\kappa=1}^{k}
\Big(P_\kappa^{R,Taylor}
+\phi_i \partial_i P_\kappa^{R,Taylor}
\Big)
\\&~~~~~~~~~~~~~~
+\sum_{\kappa=2}^k \phi_i \nabla \partial_i P_\kappa^{R,Taylor}
\bigg|^2
~dx
\\&~~~
+C(k)\dashint_{B_r} \sum_{\kappa=2}^k |\nabla^2 P_\kappa^{R,Taylor}|^2 |\phi|^2
~dx
\\&~~~
+C(k)\dashint_{B_r} \sum_{\kappa=2}^{k-1}
\big|\nabla \psi_{P_\kappa}^{R,Taylor}\big|^2 ~dx
\\&~~~
+6\dashint_{B_r} \big|\nabla \tilde \psi_{P_k^{R,Taylor}}\big|^2 ~dx
\\&~~~
+6\dashint_{B_r} | \nabla u-\nabla (u_{hom}+\phi_i \partial_i
u_{hom})|^2 ~dx
\\&~~~
+6\dashint_{B_r} |\phi_i \nabla \partial_i u_{hom}|^2 ~dx
\\&
\leq C(d,k,\lambda) 
\left(\frac{r}{R}\right)^{2k}
\left(1+\varepsilon_r^2\right)
\dashint_{B_R} |\nabla u|^2 ~dx
\\&~~~
+C(d,k) \sum_{\kappa=2}^k r^{2(\kappa-1)} ||P_\kappa^{R,Taylor}||^2
\varepsilon_{r}^2
\\&~~~
+C(d,k) \sum_{\kappa=2}^{k-1} ||P_\kappa^{R,Taylor}||^2
\max_{P \in \P{\kappa}} ||P||^{-2} \dashint_{B_r} |\nabla
\psi_{P}|^2 ~dx
\\&~~~
+C(d,k) ||P_k^{R,Taylor}||^2
\max_{P\in \Pa{k}} ||P||^{-2} \dashint_{B_r} |\nabla \tilde\psi_{P}|^2 ~dx
\\&~~~
+C(d,\lambda)\varepsilon_{R}^{2/(d+1)^2} \left(\frac{r}{R}\right)^{-d}
\dashint_{B_R} |\nabla u|^2 ~dx
\\&~~~
+C(d) r^2\varepsilon_r^2 \sup_{B_{R/4}} |\nabla^2 u_{hom}|^2.
\end{align*}
This finally yields in connection with the bounds on $\nabla^\kappa
u_{hom}$ in $B_{R/4}$ (cf.\ \eqref{EstimateDuhom}) which in particular imply
$||P_\kappa^{R,Taylor}||\leq C(d,k,\lambda) R^{1-\kappa} \left(\dashint_{B_R}
|\nabla u|^2 ~dx\right)^{1/2}$
\begin{align*}
&\dashint_{B_r} \bigg|
\nabla u
-\nabla \sum_{\kappa=1}^{k-1}
\Big(P_\kappa^{R,Taylor}
+\phi_i \partial_i P_\kappa^{R,Taylor}
+\psi_{P_\kappa^{R,Taylor}}
\Big)
\\&~~~~~~
-\nabla \Big(P_k^{R,Taylor}
+\phi_i \partial_i P_k^{R,Taylor}
+\tilde \psi_{P_k^{R,Taylor}}
\Big)
\bigg|^2
~dx
\\
&\leq
C(d,k,\lambda) 
\left(\frac{r}{R}\right)^{2k}
\left(1+\varepsilon_r^2\right)
\dashint_{B_R} |\nabla u|^2 ~dx
\\&~~~
+C(d,k,\lambda)\varepsilon_{r}^2
\sum_{\kappa=2}^k \left(\frac{r}{R}\right)^{2(\kappa-1)}
\dashint_{B_R} |\nabla
u|^2 ~dx
\\&~~~
+C(d,k,\lambda)
\dashint_{B_R} |\nabla u|^2 ~dx
\sum_{\kappa=2}^{k-1}
R^{-2(\kappa-1)}
\max_{P \in \P{\kappa}} ||P||^{-2} \dashint_{B_r} |\nabla
\psi_{P}|^2 ~dx
\\&~~~
+C(d,k,\lambda)
\dashint_{B_R} |\nabla u|^2 ~dx \times 
R^{-2(k-1)} \max_{P\in \Pa{k}}
||P||^{-2} \dashint_{B_r} |\nabla \tilde\psi_{P}|^2 ~dx
\\&~~~
+C(d,\lambda)\varepsilon_{R}^{2/(d+1)^2} \left(\frac{r}{R}\right)^{-d}
\dashint_{B_R} |\nabla u|^2 ~dx
\\&~~~
+C(d,\lambda) r^2\varepsilon_r^2 R^{-2} \dashint_{B_R} |\nabla u|^2 ~dx
\\&
\leq
C(d,k,\lambda)
\dashint_{B_R} |\nabla u|^2 ~dx
\bigg[
\left(\frac{r}{R}\right)^{2k}
\\&~~~~~~~~~~~~~~~~~~~
+\left(\varepsilon_{2,R}^{2/(d+1)^2}
+R^{-2(k-1)} \max_{P\in \Pa{k},||P||=1} \dashint_{B_R} |\nabla
\tilde\psi_{P}|^2 ~dx \right)
\left(\frac{r}{R}\right)^{-d}
\bigg],
\end{align*}
where in the last step we have used the inequality $\varepsilon_{r}^2\leq
\left(\frac{R}{r}\right)^d \varepsilon_{R}^2\leq \left(\frac{R}{r}\right)^d
\varepsilon_{R}^{2/(d+1)^2}$ and $\varepsilon_R\leq \varepsilon_{2,R}$ as well
as \eqref{EstimateGrowthHigherOrderCorrector} for $2\leq \kappa\leq k-1$.
Our new estimate now implies the desired bound.
\end{proof}

\subsection{The $C^{k-1,1}$ excess-decay estimate}

Like in the $C^{2,\alpha}$ case, we now show how the $C^{k,\alpha}$ excess-decay
estimate for the $k$th-order excess $\widetilde\Exc_k$ (cf. Lemma
\ref{CkalphaExcessDecayGeneral}) entails a $C^{k-1,1}$ excess-decay estimate for
the $(k-1)$th-order excess $\Exc_{k-1}$.
\begin{lemma}
\label{Ck-11Estimate}
Let $d\geq 2$, $k\geq 2$, and $R>0$. Assume that Theorem
\ref{ExistenceHigherOrderCorrector} holds for the orders $2$, \ldots, $k-1$,
and let $\psi_P\equiv 0$ for linear polynomials $P$ in order to simplify
notation.
For any $P\in \Pa{\kappa}$, denote by
$\tilde \psi_P$ a solution to the equation
\eqref{EquationPsiP} on the ball $B_R$ (without boundary conditions); assume
that the $\tilde\psi_P$ depend linearly on $P$.
Then there exists a constant $\varepsilon_{min}>0$ depending only on $d$, $k$,
and $\lambda$ such that the following assertion holds:

Suppose $r_0\in (0,R]$ is so large that $\varepsilon_{2,r_0}\leq
\varepsilon_{min}$ and
\begin{align*}
\sup_{r_0\leq \rho\leq R} \rho^{-(k-1)}
\left(\max_{P\in \Pa{k},||P||=1}
\dashint_{B_\rho} |\nabla \tilde \psi_P|^2 ~dx\right)^{1/2}
\leq\varepsilon_{min}
\end{align*}
hold. Let $u$ be an $a$-harmonic function on $B_R$. Then there exist
$P^R_\kappa \in \Pa{\kappa}$ ($1\leq
\kappa\leq k-1$) for which the estimate
\begin{align*}
&\dashint_{B_r} \left|\nabla u
-\nabla \sum_{\kappa=1}^{k-1} \left(P_\kappa^R + \phi_i \partial_i
P_\kappa^R + \psi_{P_\kappa^R}\right) \right|^2 ~dx
\\&
\leq C(d,k,\lambda)
\left(\frac{r}{R}\right)^{2(k-1)}
\dashint_{B_R} |\nabla u|^2 ~dx
\end{align*}
holds for any $r\in [r_0,R]$.
Furthermore, the
$P^R_\kappa$ depend linearly on $u$ and satisfy
\begin{align*}
\sum_{\kappa=1}^{k-1} R^{2(\kappa-1)} ||P_\kappa^R||^2
\leq C(d,k,\lambda) \dashint_{B_R} |\nabla u|^2 ~dx.
\end{align*}
\end{lemma}
\begin{proof}
In Lemma \ref{CkalphaExcessDecayGeneral}, fix $\alpha:=1/2$. We then easily
verify that Lemma \ref{CkalphaExcessDecayGeneral} is applicable in our
situation.
Set
$P_\kappa^R:=P_\kappa^{r_0,min}$; this implies that
the $P_\kappa^R$ depend linearly on $u$.
The estimate \eqref{EstimatePbSize} takes the form
\begin{align*}
\sum_{\kappa=1}^k R^{2(\kappa-1)} ||P_\kappa^R||^2
\leq C(d,k,\lambda) \dashint_{B_R} |\nabla u|^2 ~dx.
\end{align*}
Furthermore, applying Lemma \ref{CkalphaExcessDecayGeneral} with $r_0$ playing
the role of $r$ and $r$ playing the role of $R$, we deduce from
\eqref{EstimatePbDiff}
\begin{align*}
&
\sum_{\kappa=1}^k r^{2(\kappa-1)} 
||P_\kappa^R-P_\kappa^{r,min}||^2
\\&
\leq C(d,k,\lambda) \widetilde\Exc_k(r)
\\
&\stackrel{\eqref{kthOrderExcessDecay}}{\leq}
C(d,k,\lambda) \left(\frac{r}{R}\right)^{2(k-1)+2\alpha}
\widetilde\Exc_k(R)
\\
&\leq C(d,k,\lambda) \left(\frac{r}{R}\right)^{2(k-1)+2\alpha}
\dashint_{B_R} |\nabla u|^2 ~dx.
\end{align*}
We now estimate
\begin{align*}
&\dashint_{B_r} \bigg|\nabla u
-\nabla \sum_{\kappa=1}^{k-1} \left(P_\kappa^R + \phi_i \partial_i
P_\kappa^R + \psi_{P_\kappa^R}\right)
\bigg|^2 ~dx
\\&
\leq
3\dashint_{B_r} \bigg|\nabla u
-\nabla\sum_{\kappa=1}^{k-1} \left(P_\kappa^{r,min} + \phi_i \partial_i
P_\kappa^{r,min} + \psi_{P_\kappa^{r,min}}\right)
\\&~~~~~~~~~~~~
-\nabla \Big(P_k^{r,min} + \phi_i \partial_i P_k^{r,min}
+ \tilde \psi_{P_k^{r,min}}\Big) \bigg|^2 ~dx
\\&~~~
+3\dashint_{B_r} \left|\nabla\Big(
P_k^{r,min} + \phi_i \partial_i P_k^{r,min} + \tilde \psi_{P_k^{r,min}}
\Big)\right|^2 ~dx
\\&~~~
+3\dashint_{B_r} \bigg|\nabla
\sum_{\kappa=1}^{k-1} \Big(P_\kappa^{r,min}-P_\kappa^{R}
+\phi_i \partial_i (P_\kappa^{r,min}-P_\kappa^R)
+\psi_{P_\kappa^{r,min}-P_\kappa^R}\Big)
\bigg|^2 ~dx
\\&
\leq
3 \widetilde\Exc_k(r)
\\&~~~
+C(d,k)
||P_k^{r,min}||^2 r^{2(k-2)} \Bigg(\dashint_{B_r} |\phi|^2
+r^2 |\Id+(\nabla \phi)^t|^2 ~dx
\\&~~~~~~~~~~~~~~~~~~~~~~~~~~~~~~~~~~~~~~~~~
+r^{-2(k-2)}\max_{P\in \Pa{k}, ||P||=1} \dashint_{B_r} |\nabla \tilde \psi_P|^2
~dx \Bigg)
\\&~~~
+C(d,k) 
\sum_{\kappa=1}^{k-1} r^{2(\kappa-1)}
||P_\kappa^{r,min}-P_\kappa^R||^2
\dashint_{B_r} |\Id+(\nabla\phi)^t|^2 ~dx
\\&~~~
+C(d,k)
\sum_{\kappa=2}^{k-1} ||P_\kappa^{r,min}-P_\kappa^R||^2
\max_{P\in \Pa{\kappa},||P||=1}
\dashint_{B_r} r^{2(\kappa-2)} |\phi|^2 + |\nabla \psi_{P}|^2 ~dx
\\&
\stackrel{(\ref{EstimateGrowthHigherOrderCorrector},\ref{kthOrderExcessDecay},
\ref{CaccioppoliPhi2})}{\leq~~~~~~~~}
C(d,k,\lambda) \left(\frac{r}{R}\right)^{2(k-1)+2\alpha} \widetilde\Exc_k(R)
\\&~~~
+C(d,k,\lambda) ||P^{r,min}_k||^2 r^{2(k-1)}
(\varepsilon_r^2 + (1+\varepsilon_{2r}^2) + \varepsilon_{\tilde\psi,r_0,R}^2)
\\&~~~
+C(d,k,\lambda)
\sum_{\kappa=1}^{k-1} r^{2(\kappa-1)}||P_\kappa^{r,min}-P_\kappa^R||^2
(1+\varepsilon_{2r}^2)
\\&~~~
+C(d,k,\lambda) 
\sum_{\kappa=2}^{k-1} r^{2(\kappa-1)}||P_\kappa^{r,min}-P_\kappa^R||^2
(\varepsilon_{r}^2+\varepsilon_{2,r}^2)
\\&
\leq
C(d,k,\lambda) \left(\frac{r}{R}\right)^{2(k-1)+2\alpha} \widetilde\Exc_k(R)
+C(d,k,\lambda) ||P_k^{r,min}||^2 r^{2(k-1)}
\\&~~~
+C(d,k,\lambda) \sum_{\kappa=1}^{k-1}
r^{2(\kappa-1)}||P_\kappa^{r,min}-P_\kappa^R||^2.
\end{align*}
In conjunction with the two previous estimates, we infer
\begin{align*}
&\dashint_{B_r} \bigg|\nabla u
-\nabla \sum_{\kappa=1}^{k-1} \left(P_\kappa^R + \phi_i \partial_i
P_\kappa^R + \psi_{P_\kappa^R}\right)
\bigg|^2 ~dx
\\&
\leq
C(d,k,\lambda) \bigg[\left(\frac{r}{R}\right)^{2(k-1)+2\alpha}
+\left(\left(\frac{r}{R}\right)^{2(k-1)}
+\left(\frac{r}{R}\right)^{2(k-1)+2\alpha}\right)
+\left(\frac{r}{R}\right)^{2(k-1)+2\alpha}
\bigg]
\\&~~~~~~\times
\dashint_{B_R} |\nabla u|^2~dx.
\end{align*}
Our lemma is therefore established.
\end{proof}

\subsection{Construction of correctors of order $k$}

Using the $C^{k-1,1}$ theory established in the previous subsection, we now
proceed to the construction of our $k$th-order corrector. The following lemma
provides the inductive step; starting from a function which acts as a
$k$th-order corrector on a ball $B_R$, we construct a function acting as a
$k$th-order corrector on the ball $B_{2R}$.

\begin{lemma}
\label{IterativeConstructionkthCorrector}
Let $d\geq 2$, $k\geq 2$, and assume that Theorem
\ref{ExistenceHigherOrderCorrector} holds for the orders $2$, \ldots, $k-1$.
Let $r_0>0$ satisfy the estimate $\varepsilon_{2,r_0}\leq \varepsilon_0$, where
$\varepsilon_0=\varepsilon_0(d,k,\lambda)$ is to be chosen in the proof below.
Then the following implication holds:

Let $R=2^M r_0$ for some $M\in \mathbb{N}_0$.
Suppose that for every $P\in \P{k}$ we have a solution $\psi^R_P$ to the
equation
\begin{align*}
-\divv a\nabla \psi_P^R = \divv (\chi_{B_R} (\phi_i a-\sigma_i)\nabla \partial_i
P)
\end{align*}
subject to the growth condition
\begin{align*}
r^{-(k-1)}
\left(
\dashint_{B_r} |\nabla \psi_P^R|^2 ~dx
\right)^{1/2}
\leq C_1(d,k,\lambda) ||P||
\sum_{m=0}^{M}
\min\{1,2^m r_0/r\}\varepsilon_{2^m r_0}
\end{align*}
for all $r\geq r_0$,
where $C_1(d,k,\lambda)$ is a sufficiently large constant to be chosen in the
proof below. Assume furthermore that $\psi^R_P$ depends linearly on $P$.

Then for every $P\in \P{k}$ there exists a solution $\psi^{2R}_P$ to
the equation
\begin{align*}
-\divv a\nabla \psi_P^{2R} = \divv (\chi_{B_{2R}} (\phi_i a-\sigma_i)\nabla
\partial_i P)
\end{align*}
subject to the growth condition
\begin{align*}
r^{-(k-1)}
\left(
\dashint_{B_r} |\nabla \psi_P^{2R}|^2 ~dx
\right)^{1/2}
\leq C_1(d,k,\lambda) ||P||
\sum_{m=0}^{M+1}
\min\{1,2^m r_0/r\}\varepsilon_{2^m r_0}
\end{align*}
for all $r\geq r_0$.
Furthermore, $\psi_P^{2R}$ depends linearly on $P$ and we have
\begin{align*}
&r^{-(k-1)}
\left(\dashint_{B_r} |\nabla \psi_P^{2R}-\nabla \psi_P^R|^2
~dx\right)^{1/2}
\leq C_1(d,k,\lambda) ||P||
\varepsilon_{2^{M+1} r_0}.
\end{align*}
\end{lemma}
\begin{proof}
To establish the lemma, we first note that the assumptions of the lemma ensure
that the $C^{k-1,1}$ excess-decay lemma (Lemma \ref{Ck-11Estimate}) is
applicable on $B_R$ with $\tilde \psi_P:=\psi_P^{R}$. To see this, we estimate
for any $r\in [r_0,R]$
\begin{align*}
r^{-(k-1)}\left(
\dashint_{B_r} |\nabla \psi_P^R|^2 ~dx
\right)^{1/2}
\leq C_1(d,k,\lambda) ||P||
\varepsilon_{2,r_0}
\leq C_1(d,k,\lambda) ||P|| \varepsilon_0.
\end{align*}
By choosing $\varepsilon_0>0$ small enough depending only on $d$, $k$,
$\lambda$, and $C_1$ (which is to be chosen at the end of this proof), we can
ensure that the assumption of Lemma~\ref{Ck-11Estimate} regarding smallness of
$\varepsilon_{\tilde \psi,r_0,R}$ is satisfied.

We now turn to the construction of $\psi_P^{2R}-\psi_P^R$ and to that purpose
denote by $\xi_P^R$ the weak solution on $\mathbb{R}^d$ with zero
mean in $B_{2R}$ and square integrable gradient, whose existence and uniqueness
follows by the Lax-Milgram theorem, to the problem
\begin{align*}
-\divv a\nabla \xi_P^R
=&
\divv (\chi_{B_{2R}-B_R} (\phi_i a-\sigma_i)\nabla \partial_i P).
\end{align*}
Obviously, $\xi_P^R$ depends linearly on $P$. Furthermore, by ellipticity we
have the estimate
\begin{align*}
&\int_{\mathbb{R}^d} |\nabla \xi_P^R|^2 ~dx
\\&
\leq C(d,\lambda) \sup_{B_{2R}} |\nabla^2 P|
\left(\int_{\mathbb{R}^d} \chi_{B_{2R}-B_R}(|\phi a|^2+|\sigma|^2) ~dx
\right)^{1/2} \left(\int_{\mathbb{R}^d} |\nabla \xi_P^R|^2 ~dx\right)^{1/2}
\end{align*}
which gives
\begin{align*}
&\left(\int_{\mathbb{R}^d} |\nabla \xi_P^R|^2 ~dx\right)^{1/2}
\leq C(d,\lambda)
\sup_{B_{2R}} |\nabla^2 P|
\left(\int_{B_{2R}} |\phi|^2 + |\sigma|^2 ~dx \right)^{1/2}.
\end{align*}
The last estimate in turn implies
\begin{align}
\label{EstimateXiRP}
\int_{\mathbb{R}^d} |\nabla \xi_P^R|^2 ~dx
\leq C(d,k,\lambda) ||P||^2 R^{2(k-2)} \varepsilon_{2R}^2 R^{2+d}.
\end{align}
We now obtain $\psi_P^{2R}-\psi_P^R$ by modifying $\xi_P^R$ by an $a$-harmonic
function of degree $k-1$.
As $\xi_P^R$ is $a$-harmonic in $B_R$, Lemma \ref{Ck-11Estimate} now
implies the existence of
some $P_{\kappa,P}^R\in \P{\kappa}$ for $1\leq \kappa\leq k-1$ which depend
linearly on $P$ and for which the estimates
\begin{align}
\label{EstimatebRP}
||P_{\kappa,P}^R||^2\leq 
C(d,k,\lambda) R^{-2(\kappa-1)} \dashint_{B_R} |\nabla \xi_P^R|^2 ~dx
\stackrel{(\ref{EstimateXiRP})}{\leq}
C(d,k,\lambda) ||P||^2 R^{2(k-\kappa)} \varepsilon_{2R}^2
\end{align}
and
\begin{align*}
&\dashint_{B_r} \bigg|\nabla \xi_P^R
-\nabla \sum_{\kappa=1}^{k-1}\Big(P_{\kappa,P}^R+\phi_i \partial_i
P_{\kappa,P}^R +\psi_{P_{\kappa,P}^R}\Big)
\bigg|^2 ~dx
\\&
\stackrel{~~~~}{\leq} C(d,k,\lambda)
\left(\frac{r}{R}\right)^{2(k-1)}
\dashint_{B_R} |\nabla \xi_P^R|^2 ~dx
\\&
\stackrel{(\ref{EstimateXiRP})}{\leq}
C(d,k,\lambda) ||P||^2 r^{2(k-1)} \varepsilon_{2R}^2
\end{align*}
hold for all $r\in [r_0,R]$.

Furthermore, we have for $r>R$
\begin{align*}
&\dashint_{B_r} \bigg|\nabla \xi_P^R
-\nabla \sum_{\kappa=1}^{k-1}\Big(P_{\kappa,P}^R+\phi_i \partial_i
P_{\kappa,P}^R +\psi_{P_{\kappa,P}^R}\Big)
\bigg|^2 ~dx
\\&
\stackrel{(\ref{CaccioppoliPhi2},\ref{EstimateGrowthHigherOrderCorrector})}{\leq}
C(d,k,\lambda) \Bigg(r^{-d}\int_{B_r} |\nabla \xi_P^R |^2 ~dx
+||P_{1,P}^R||^2 (1+\varepsilon_{2r}^2)
\\&~~~~~~~~~~~~~~~~~~~~~~~~
+\sum_{\kappa=2}^{k-1} r^{2(\kappa-1)}
||P_{\kappa,P}^R||^2 (1+\varepsilon_{2r}^2+\varepsilon_{2,r}^2)
\Bigg)
\\&
\stackrel{(\ref{EstimateXiRP},\ref{EstimatebRP})}{\leq}
C(d,k,\lambda) ||P||^2 R^{2(k-1)}
\left(\left(\frac{R}{r}\right)^d
+1+\varepsilon_{2r}^2
+(1+\varepsilon_{2r}+\varepsilon_{2,r})\left(\frac{r}{R}\right)^{2(k-2)}
\right)\varepsilon_{2R}^2
\\&
\stackrel{~~~~~~~\,}{\leq} C(d,k,\lambda) ||P||^2 r^{2(k-2)} R^2
\varepsilon_{2R}^2.
\end{align*}
The combination of both $r$-ranges yields
\begin{align}
\label{LastContributionCk-11}
&\frac{1}{r^{k-1}}
\left(\dashint_{B_r} \bigg|\nabla \xi_P^R
-\nabla\sum_{\kappa=1}^{k-1}\Big(P_{\kappa,P}^R+\phi_i \partial_i P_{\kappa,P}^R
+\psi_{P_{\kappa,P}^R}
\Big)
\bigg|^2 ~dx
\right)^{1/2}
\\&
\nonumber
\leq C(d,k,\lambda)  ||P|| \min\{1,2R/r\} \varepsilon_{2R}.
\end{align}
In total, we see that
\begin{align*}
\psi_P^{2R}:=\psi_P^R+\xi_P^R
-\sum_{\kappa=1}^{k-1}\Big(P_{\kappa,P}^R+\phi_i \partial_i P_{\kappa,P}^R
+\psi_{P_{\kappa,P}^R}\Big)
\end{align*}
is the desired function (note in particular that the last term is
$a$-harmonic), provided we choose $C_1$ to be the constant appearing in
\eqref{LastContributionCk-11}.
\end{proof}

We now establish existence of $k$th-order correctors by means of the previous
lemma.
\begin{proof}[Proof of Theorem \ref{ExistenceHigherOrderCorrector}]
We just need to construct an ``initial'' $k$th-order corrector
$\psi_{P}^{r_0}$ subject to the properties of Lemma
\ref{IterativeConstructionkthCorrector}; then Lemma
\ref{IterativeConstructionkthCorrector} yields a sequence
$(\psi_P^{2^m r_0})_m$ which (after subtracting appropriate constants)
is a Cauchy sequence in $H^1(B_R)$ for every $R>0$
due to the last estimate in the lemma and our assumption \eqref{Epsilon2Small}
which implies summability of $\varepsilon_{2^m r_0}$.
Thus, the limit $\psi_P$ satisfies the equation \eqref{EquationPsiP} in the
whole space, depends linearly on $P$, and satisfies the estimate
\begin{align*}
r^{-(k-1)}
\left(
\dashint_{B_r} |\nabla \psi_P|^2 ~dx
\right)^{1/2}
&\leq C_1(d,k,\lambda) ||P||
\sum_{m=0}^{\infty}
\min\{1,2^m r_0/r\} \varepsilon_{2^m r_0}
\\&
\leq C_1(d,k,\lambda) ||P|| \varepsilon_{2,r}
\end{align*}
for any $r\geq r_0$.

To construct $\psi_P^{r_0}$, we use Lax-Milgram to find the (unique) solution
$\psi_P^{r_0}$ on $\mathbb{R}^d$ with square-integrable gradient and zero mean
on $B_{r_0}$ to the equation
\begin{align*}
-\divv a\nabla \psi_P^{r_0} = 
\divv (\chi_{B_{r_0}} (\phi_i a-\sigma_i)\nabla \partial_i
P).
\end{align*}
Obviously, $\psi_P^{r_0}$ depends linearly on $P$.
Furthermore, we have the energy estimate
\begin{align*}
&\int_{\mathbb{R}^d} |\nabla \psi_P^{r_0}|^2 ~dx
\\&
\leq C(d,\lambda) \sup_{B_{r_0}} |\nabla^2 P|
\left(\int_{\mathbb{R}^d}  |\chi_{B_{r_0}} a \phi|^2 + |\chi_{B_{r_0}}\sigma|^2
~dx\right)^{1/2} \left(\int_{\mathbb{R}^d} |\nabla \psi_P^{r_0}|^2
~dx\right)^{1/2}.
\end{align*}
We therefore get
\begin{align*}
&\left(\int_{\mathbb{R}^d} |\nabla \psi_P^{r_0}|^2 ~dx\right)^{1/2}
\leq C(d,\lambda)
\sup_{B_{r_0}} |\nabla^2 P|
\left(\int_{B_{r_0}} |\phi|^2+|\sigma|^2 ~dx\right)^{1/2}.
\end{align*}
This yields in particular for any $r\geq r_0$
\begin{align*}
\int_{B_r} |\nabla \psi_P^{r_0}|^2 ~dx
\leq C(d,k,\lambda) ||P||^2 r_0^{2(k-2)} \int_{B_{r_0}} |\phi|^2+|\sigma|^2 ~dx
\end{align*}
and therefore
\begin{align*}
\dashint_{B_r} |\nabla \psi_P^{r_0}|^2 ~dx
&\leq C(d,k,\lambda) ||P||^2 r^{-d} r_0^{2(k-2)} \varepsilon_{r_0}^2 r_0^{2+d}
\\& 
\leq C(d,k,\lambda) ||P||^2 r^{2(k-1)} \min\{1,(r_0/r)^2\}
\varepsilon_{r_0}^2.
\end{align*}
We note that this provides the starting point for Lemma
\ref{IterativeConstructionkthCorrector}, possibly after enlarging the
constant $C_1$ in the statement thereof.
\end{proof}

\subsection{Proof of the $k$th-order Liouville principle}

Like in the $C^{2,\alpha}$ case, the $C^{k,\alpha}$ Liouville principle (Lemma
\ref{LiouvillePrincipleCkalpha} below) is an easy consequence of our large-scale
excess-decay estimate (Theorem~\ref{CkalphaExcessDecay}). The $k$th-order
Liouville principle (Corollary \ref{LiouvillePrincipleK}) in turn is an easy
consequence of the $C^{k+1,\alpha}$ Liouville principle.

\begin{lemma}
\label{LiouvillePrincipleCkalpha}
Let $d\geq 2$, $k\geq 2$, and suppose that the assumption (\ref{Epsilon2Small})
is satisfied.
Then the following property holds: Any $a$-harmonic function $u$
satisfying the growth condition
\begin{align}
\label{GrowthConditionInkalphaLiouville}
\liminf_{r\rightarrow \infty} \frac{1}{r^{k+\alpha}}
\left(\dashint_{B_r} \left|u\right|^2 ~dx\right)^{1/2}=0
\end{align}
for some $\alpha\in (0,1)$ is of the form
\begin{align*}
u=a+b_i (x_i+\phi_i)+\sum_{\kappa=2}^k (P_\kappa+\phi_i \partial_i
P_\kappa+\psi_{P_\kappa})
\end{align*}
with some $a\in \mathbb{R}$, $b\in \mathbb{R}^d$, and $P_\kappa\in
\Pa{\kappa}$ for $2\leq \kappa\leq k$ (i.e. $P_\kappa$ is a homogeneous
$a_{hom}$-harmonic polynomial of degree $\kappa$). Here, the $\psi_{P}$ denote
the higher-order correctors whose existence is guaranteed by Theorem
\ref{ExistenceHigherOrderCorrector}.
\end{lemma}

\begin{proof}[Proof of Corollary \ref{LiouvillePrincipleK}]
Obviously, \eqref{GrowthConditionInkLiouville} entails
\eqref{GrowthConditionInkalphaLiouville} with $k+1$ in place of $k$ and e.g.
$\alpha:=\frac{1}{2}$. By Lemma \ref{LiouvillePrincipleCkalpha}, any
$a$-harmonic function $u$ subject to condition
\eqref{GrowthConditionInkLiouville} must be of the form
\begin{align}
\label{RepresentationULiouville}
u=a+b_i (x_i+\phi_i)+\sum_{\kappa=2}^{k+1} (P_\kappa+\phi_i \partial_i
P_\kappa+\psi_{P_\kappa})
\end{align}
with some $a\in \mathbb{R}$, $b\in \mathbb{R}^d$, and $P_\kappa\in
\Pa{\kappa}$ for $2\leq \kappa\leq k+1$. Our stronger growth condition
\eqref{GrowthConditionInkLiouville} however shows that we have
$P_{k+1}\equiv 0$: Since the $\phi_i$ grow sublinearly, cf.\
\eqref{EpsilonSmall}, and since $\psi_{P_{k+1}}$ grows slower than a polynomial
of degree $k+1$, cf.\ \eqref{EstimateGrowthHigherOrderCorrector}, we see that for
large $|x|$ the term $P_{k+1}$ would be the dominating term in
\eqref{RepresentationULiouville} if it were nonzero, contradicting our growth
condition \eqref{GrowthConditionInkLiouville}.
\end{proof}

\begin{proof}[Proof of Lemma \ref{LiouvillePrincipleCkalpha}]
Let $\alpha\in (0,1)$ be such that
\begin{align*}
\liminf_{R\rightarrow \infty} \frac{1}{R^{k+\alpha}}
\left(\dashint_{B_R} \left|u\right|^2 ~dx\right)^{1/2}=0
\end{align*}
holds. By the Caccioppoli estimate, we deduce
\begin{align*}
\liminf_{R\rightarrow \infty} \frac{1}{R^{k-1+\alpha}}
\left(\dashint_{B_R} \left|\nabla u\right|^2 ~dx\right)^{1/2}=0.
\end{align*}
Fix $r\geq r_0$. The excess-decay estimate from Theorem \ref{CkalphaExcessDecay}
together with the trivial bound $\Exc_k(R)\leq \dashint_{B_R} |\nabla
u|^2 ~dx$ yields
\begin{align*}
\Exc_k(r)&\leq C(d,k,\lambda,\alpha)\left(\frac{r}{R}\right)^{2(k-1)+2\alpha}
\Exc_k(R)
\\&
\leq C(d,k,\lambda,\alpha) r^{2(k-1)+2\alpha} \left(\frac{1}{R^{k-1+\alpha}}
\left(\dashint_{B_R} |\nabla u|^2 ~dx\right)^{1/2}\right)^2.
\end{align*}
Passing to the lim inf $R\rightarrow \infty$, we deduce that
\begin{align*}
\Exc_k(r)=0
\end{align*}
holds for every $r\geq r_0$. Therefore, on every $B_r$ with $r\geq r_0$, $\nabla
u$ can be represented \emph{exactly} as the derivative of a corrected polynomial
of $k$th order (since the infimum in the definition of $\Exc_k$ is actually
attained, as noted at the beginning of the proof of Lemma
\ref{ExcessDecayLemmak}), i.e.
we have
\begin{align*}
\nabla u=\nabla b_i^r (x_i+\phi_i) + \nabla
\sum_{\kappa=2}^k (P_\kappa^r+\phi_i \partial_i
P_\kappa^r+\psi_{P_\kappa^r})
\end{align*}
in $B_r$ for some $b^r\in \mathbb{R}^d$ and some $P_\kappa^r\in \Pa{\kappa}$
($2\leq \kappa\leq k$); recall that we have used the convention $\psi_P\equiv
0$ for linear polynomials $P$.
It is not difficult to show that for $r$ large enough, the $b^r$ and
$P_\kappa^r$ are actually independent of $r$ and define some common $b\in
\mathbb{R}^d$ and $P_\kappa\in \Pa{\kappa}$: For example, one may use Lemma
\ref{CkalphaExcessDecayGeneral} to compare the $b^r$, $P_\kappa^r$ for two
different radii $r_1,r_2\geq r_0$; the estimate for $|b^{r_1}-b^{r_2}|$ and
$||P_\kappa^{r_1}-P_\kappa^{r_2}||$ then contains the factor
$\Exc_k(\max(r_1,r_2))$ and is therefore zero.
Moreover, the gradient $\nabla u$ determines the function $u$ itself up to a
constant, i.e. we have
\begin{align*}
u=a + b_i (x_i+\phi_i) +\sum_{\kappa=2}^k (P_\kappa+\phi_i \partial_i
P_\kappa+\psi_{P_\kappa})
\end{align*}
for some $a\in \mathbb{R}$, $b\in \mathbb{R}^d$, and $P_\kappa\in \Pa{\kappa}$
($2\leq \kappa\leq k$).
\end{proof}


\appendix

\section{Approximation of $a$-Harmonic Functions by Corrected $a_{hom}$-Harmonic
Functions}

Our proofs make use of the following lemma, which is implicitly derived in the
course of the proof of Lemma 2 in \cite{GloriaNeukammOtto}. For the reader's
convenience, we recall its proof here.

The lemma essentially states that an $a$-harmonic function $u$ on a ball $B_R$
may be approximated on the ball $B_{R/2}$ up to a small error (of order
$\varepsilon^{1/(d+1)^2}_R$) by an appropriate $a_{hom}$-harmonic
function $u_{hom}$ and correcting this function $u_{hom}$ using the first-order
corrector $\phi_i$.

The purpose of the lemma is the same as in classical elliptic regularity theory:
The function $u_{hom}$ satisfies an elliptic equation with constant
coefficients, i.e. it is smooth and good estimates for its higher derivatives
are available. In our proof above, we show by means of the present lemma that
this high regularity of $u_{hom}$ transfers (in an appropriate sense)
to $u$ itself.
\begin{lemma}
\label{ApproximationConstantCoefficient}
Let $R>0$ and let $u$ be $a$-harmonic on $B_R$. Suppose that
$\varepsilon_R\leq 1$ (with $\varepsilon_R$ as defined in
\eqref{DefinitionEpsilon}). Then there exists an $a_{hom}$-harmonic function
$u_{hom}$ on $B_{R/2}$ satisfying the following two properties: First, we have
the energy estimate
\begin{align}
\label{EnergyEstimateuhom}
\dashint_{B_{R/2}} |\nabla u_{hom}|^2 ~dx
\leq C(d,\lambda)\dashint_{B_R} |\nabla u|^2 ~dx.
\end{align}
Second, the ``corrected'' function $u_{hom}+\phi_i \partial_i u_{hom}$ is a good
approximation for $u$ in the sense that
\begin{align*}
\dashint_{B_{R/2}} |\nabla u-\nabla (u_{hom}+\phi_i \partial_i u_{hom})|^2 ~dx
\leq
C(d,\lambda)\varepsilon_{R}^{2/(d+1)^2}
\dashint_{B_R} |\nabla u|^2 ~dx.
\end{align*}
\end{lemma}
\begin{proof}
Choose some $R'\in [\frac{3}{4}R,R]$ for which
\begin{align}
\label{AssumptionBoundary}
R'\dashint_{\partial B_{R'}} |\nabla u|^2 ~dS
\leq C(d) \dashint_{B_R} |\nabla u|^2 ~dx
\end{align}
holds. Let $u_{hom}$ be the $a_{hom}$-harmonic function in $B_{R'}$
which coincides with $u$ on $\partial B_{R'}$. Testing the equation $-\divv
a_{hom} \nabla u_{hom}=0$ with $u_{hom}-u$ (note that this test function is
admissible since we have $u_{hom}-u=0$ on $\partial B_{R'}$), we infer by
ellipticity of $a$ and (in the second step) Young's inequality
\begin{align}
\nonumber
\dashint_{B_{R'}} |\nabla u_{hom}|^2 ~dx
&\leq C(\lambda) \dashint_{B_{R'}} |\nabla u| |\nabla u_{hom}| ~dx
\\&
\label{EnergyEstimateuhom2}
\leq \frac{1}{2}\dashint_{B_{R'}} |\nabla u_{hom}|^2 ~dx
+C(\lambda) \dashint_{B_{R'}} |\nabla u|^2 ~dx,
\end{align}
which because of $R/2\leq R'\leq R$ gives the desired energy estimate.
It remains to establish the approximation property of $u_{hom}+\phi_i \partial_i
u_{hom}$.

Denote by
$\eta_0:\mathbb{R}\rightarrow \mathbb{R}$ a smooth function with $\eta_0(s)=1$
for $s\geq 1$ and $\eta_0(s)=0$ for $s\leq 0$. Let $0<\rho<R/4$ and set
$\eta(x):=\eta_0(2(R'-\rho/2-|x|)/\rho)$. Note that we have $|\nabla \eta|\leq
C(d)/\rho$ as well as $\eta\equiv 0$ outside of $B_{R'-\rho/2}$ and $\eta \equiv
1$ in $B_{R'-\rho}$. Due to $\rho\leq R/4$, we also have $R'-\rho\geq R/2$. We
will optimize in this ``boundary layer thickness'' $\rho$ at the end of the
proof.

Let us abbreviate
\begin{align*}
v:=u-u_{hom}-\eta \phi_i \partial_i u_{hom}.
\end{align*}
where the purpose of $\eta$ is to have $v\equiv 0$ on $\partial B_{R'}$.
The desired approximation
property of $u_{hom}+\phi_i \partial_i u_{hom}$ as stated in the lemma will be a
consequence of an appropriate energy estimate for $v$ (recall that we have
$\eta\equiv 1$ in $B_{R/2}$ since $\rho<R/4$ and $R'>3R/4$).

To derive this energy estimate, we would like to show that $v$ is approximately
$a$-harmonic. We first compute using the fact that $u$ and $x_i+\phi_i$ are
$a$-harmonic (cf. \eqref{EquationCorrector})
\begin{align*}
&-\divv a \nabla v
\\
&=-\divv a \nabla u + \divv (1-\eta) a\nabla u_{hom}
+\divv a (e_i+\nabla \phi_i) \eta \partial_i u_{hom}
+\divv \phi_i a \nabla (\eta \partial_i u_{hom})
\\
&\stackrel{\eqref{EquationCorrector}}{=}
\divv (1-\eta) a \nabla u_{hom}
+a (e_i+\nabla \phi_i) \cdot \nabla (\eta \partial_i u_{hom})
+\divv \phi_i a \nabla (\eta \partial_i u_{hom})
\\
&=
\divv (1-\eta) (a-a_{hom}) \nabla u_{hom}
+(a(e_i+\nabla \phi_i)-a_{hom}e_i) \cdot \nabla (\eta \partial_i u_{hom})
\\&~~~
+\divv \phi_i a \nabla (\eta \partial_i u_{hom}),
\end{align*}
where in the last step we have used the $a_{hom}$-harmonicity of $u_{hom}$ in
form of the equality $-\divv (1-\eta)a_{hom}\nabla
u_{hom}-a_{hom}e_i \cdot \nabla(\eta\partial_i u_{hom})=0$. Taking into account
the formula $a(e_i+\nabla\phi_i)-a_{hom}e_i=\divv \sigma_i$ (cf.
\eqref{Equationq},\eqref{EquationSigma}) and the fact that
\begin{align*}
(\nabla \cdot \sigma_i)\cdot \nabla w
=\partial_k \sigma_{ijk} \partial_j w 
=\partial_k (\sigma_{ijk}\partial_j w)
=-\partial_k (\sigma_{ikj}\partial_j w)
=-\divv (\sigma_i \nabla w)
\end{align*}
holds for any function $w$ by skew-symmetry of $\sigma_i$, we may rewrite the
right-hand side in divergence form:
\begin{align*}
&-\divv a \nabla v=
\divv (1-\eta) (a-a_{hom}) \nabla u_{hom}
+\divv (\phi_i a-\sigma_i) \nabla (\eta \partial_i u_{hom}).
\end{align*}
Testing the weak formulation of this equation with $v$ (recall that $v\equiv 0$
on $\partial B_{R'}$) and using the ellipticity of $a$, we deduce using Young's
inequality and the properties of $\eta$
\begin{align*}
&\int_{B_{R'}} |\nabla v|^2 ~dx
\\&
\leq C(\lambda) \int_{B_{R'}} |(1-\eta)(a-a_{hom})\nabla u_{hom}|^2+|\phi_i a -
\sigma_i|^2 |\nabla (\eta \partial_i u_{hom})|^2 ~dx
\\&
\leq C(d,\lambda) \int_{B_{R'}} |1-\eta|^2 |\nabla u_{hom}|^2 ~dx
\\&~~~
+C(d,\lambda) \int_{B_{R'}} (|\phi|^2+|\sigma|^2)
(|\nabla \eta|^2 |\nabla u_{hom}|^2 + \eta^2 |\nabla^2 u_{hom}|^2) ~dx
\\&
\leq C(d,\lambda) \int_{B_{R'}-B_{R'-\rho}} |\nabla u_{hom}|^2 ~dx
\\&~~~
+C(d,\lambda)
\sup_{B_{R'-\rho/2}} \left(\frac{1}{\rho^2}|\nabla u_{hom}|^2+|\nabla^2
u_{hom}|^2\right) \int_{B_{R'}} |\phi|^2+|\sigma|^2 ~dx.
\end{align*}
Since our function $u_{hom}$ is $a_{hom}$-harmonic, we have the regularity
estimates
\begin{align*}
\sup_{B_{R'-\rho/2}}
\left(\frac{1}{\rho^2}|\nabla u_{hom}|^2+|\nabla^2 u_{hom}|^2\right)
&\leq \frac{C(d,\lambda)}{\rho^2}
\sup_{y\in B_{R'-\rho/2}} \dashint_{B_{\rho/2}(y)} |\nabla u_{hom}|^2 ~dx,
\\
\left(\int_{B_{R'}} |\nabla u_{hom}|^p ~dx\right)^{2/p}
&\leq C(d,\lambda)
\int_{\partial B_{R'}} |\nabla^{tan} u_{hom}|^2 ~dS,
\end{align*}
where $p:=2d/(d-1)$: The first estimate is a standard constant-coefficient
interior regularity estimate (which is a consequence e.g. of an iterative
application of Theorem 4.9 in \cite{GiaquintaMartinazzi} and the Sobolev
embedding).
The second estimate follows by combining
1) the existence of an extension $\bar
u$ of $u_{hom}$ subject to the estimate $||\nabla\bar u||_{L^p(B_{R'})}\leq
C(d)||\nabla^{tan}u_{hom}||_{L^2(\partial B_{R'})}$ and 2) the Calderon-Zygmund
estimate on $B_{R'}$, which reads $||\nabla w||_{L^p(B_{R'})}\leq C(d,\lambda)
||\nabla \bar u||_{L^p(B_{R'})}$ for any solution $w\in H^1(B_{R'})$ with
$w-\bar u\in H^1_0(B_{R'})$ to the equation $-\divv a_{hom}\nabla w=0$. For the
latter estimate, see Theorem 7.1 in \cite{GiaquintaMartinazzi}.

Using these regularity estimates, the equality $\nabla^{tan}u_{hom}=\nabla^{tan}
u$ on $\partial B_R$, as well as the obvious inequality
\begin{align*}
\sup_{y\in B_{R'-\rho/2}} \dashint_{B_{\rho/2}(y)} |\nabla u_{hom}|^2 ~dx
&\leq
\left(\frac{2R'}{\rho}\right)^d \dashint_{B_{R'}} |\nabla u_{hom}|^2 ~dx,
\end{align*}
we infer by $\rho\leq R'/4$ and $3R/4\leq R'\leq R$
\begin{align*}
\int_{B_{R'}} |\nabla v|^2 ~dx
&\leq C(d,\lambda) |B_{R'}-B_{R'-\rho}|^{1-2/p}
\left(\int_{B_{R'}-B_{R'-\rho}} |\nabla u_{hom}|^p ~dx\right)^{2/p}
\\&~~~
+C(d,\lambda) \frac{1}{{R'}^2} \left(\frac{R'}{\rho}\right)^{d+2}
\dashint_{B_{R'}} |\nabla u_{hom}|^2 ~dx
\cdot (R')^d\dashint_{B_{R'}} |\phi|^2+|\sigma|^2 ~dx
\\
&\stackrel{\eqref{EnergyEstimateuhom2}}{\leq}
C(d,\lambda) \rho^{1/d} {R'}^{(d-1)/d}
\int_{\partial B_{R'}} |\nabla^{tan} u|^2 ~dS
\\&~~~
+C(d,\lambda) \varepsilon_{R}^2 \left(\frac{R'}{\rho}\right)^{d+2}
\int_{B_{R'}} |\nabla u|^2 ~dx
\\&
\stackrel{\eqref{AssumptionBoundary}}{\leq}
C(d,\lambda) \left(\frac{\rho}{R'}\right)^{1/d}
\int_{B_R} |\nabla u|^2 ~dx
\\&~~~
+C(d,\lambda) \varepsilon_{R}^2 \left(\frac{R'}{\rho}\right)^{d+2}
\int_{B_{R'}} |\nabla u|^2 ~dx.
\end{align*}
We optimize in $\rho$ by choosing
$\rho:=\frac{1}{4}\varepsilon_{R}^{2d/(d+1)^2}R'$ (which thanks to the
assumption $\varepsilon_R\leq 1$ is admissible in the sense of $\rho\leq
\frac{1}{4}R'$). This yields
\begin{align*}
\int_{B_{R'}} |\nabla v|^2 ~dx
\leq C(d,\lambda) \varepsilon_{R}^{2/(d+1)^2}
\left(\int_{B_{R'}} |\nabla u|^2 ~dx
+\int_{B_{R}} |\nabla u|^2 ~dx\right)
\end{align*}
which together with
the estimate $3R/4\leq R'\leq R$ and $\eta\equiv 1$ in $B_{R/2}$ proves the
desired approximation result.
\end{proof}

\section{Failure of Liouville principle for smooth uniformly elliptic
coefficient fields}

We now provide the argument that smoothness of a uniformly elliptic coefficient
field does not prevent Liouville's theorem from failing: Even for smooth
uniformly elliptic coefficient fields, sublinearly growing harmonic functions
are not necessarily constant, implying a failure even of the zero-th order
Liouville theorem.
\begin{proposition}
\label{SmoothLiouvilleCounterexample}
For any $\alpha\in (0,1)$ there exists a smooth, bounded, and uniformly elliptic
symmetric coefficient field $a$ on $\mathbb{R}^2$ such that the following holds:
There exists a smooth function $u$ which is $a$-harmonic and satisfies
\begin{equation}\label{a.1}
\Big(\dashint_{B_R}u^2~dx\Big)^\frac{1}{2}\sim R^\alpha\quad\mbox{for}\;R\gg 1.
\end{equation}
\end{proposition}
\begin{proof}
By a classical example in dimension $d=2$ (cf.\ e.g.\ \cite{PiccininiSpagnolo}),
for any exponent $\alpha\in(0,1)$,
there exists a uniformly elliptic, symmetric coefficient field $a_0$ of a
scalar equation and a weakly $a_0$-harmonic function $u_0$ (in particular, it
is locally integrable and of locally integrable gradient) whose modulus on
average grows like $|x|^\alpha$, for instance as expressed by
\begin{equation}\label{a.2}
\Big(\dashint_{B_R}u_0^2 ~dx\Big)^\frac{1}{2}\sim R^\alpha.
\end{equation}
Moreover, in this example
\begin{equation}\label{a.7}
a_0\;\mbox{and}\;u_0\;\mbox{are homogeneous and smooth outside the origin}.
\end{equation}
We now argue that this example may be post-processed to an
example of an {\it everywhere smooth} uniformly elliptic symmetric coefficient
field $a$ and a smooth $a$-harmonic function $u$ such that still
\eqref{a.1} holds.

Indeed, because of (\ref{a.7}) we can easily construct a uniformly elliptic
coefficient field $a$ that agrees with $a_0$ outside of $B_1$ and is smooth.
Next we observe that (\ref{a.7}) also implies (using $d=2$ and $\alpha>0$) that
$\nabla u_0$ is locally {\it square} integrable, so that by Riesz'
representation theorem, there exists a weak solution of
\begin{equation}\label{a.8}
-\nabla\cdot a\nabla w=\nabla\cdot(a-a_0)\nabla u_0
\end{equation}
in the sense that $w$ and its gradient are locally integrable and that 
\begin{equation}\label{a.3}
\int|\nabla w|^2~dx\leq C(\lambda).
\end{equation}
Equation (\ref{a.8}) is made such that $u=u_0+w$ is a weak solution
(i.\ e.\ locally integrable with locally integrable gradient) of
\begin{equation}\nonumber
-\nabla\cdot a\nabla u=0,
\end{equation}
and thus smooth since $a$ is smooth by classical uniqueness and regularity
results.
It remains to give the argument in favor of (\ref{a.1}), which in
view of (\ref{a.2}) follows once we show that (\ref{a.3}) implies in particular
for large $R$
\begin{equation}\label{a.6}
\Big(\dashint_{B_R}w^2~dx\Big)^\frac{1}{2} = o(R^\alpha).
\end{equation}
This is a well-known argument related to ``bounded mean oscillation'': By
Poincar\'e's estimate with mean value zero we have on every dyadic ball around
the origin
\begin{equation}\nonumber
\left(\int_{B_{2^n}}(w-\av_{B_{2^n}}w)^2~dx\right)^\frac{1}{2}
\leq C(d) \cdot 2^n\left(\int_{B_{2^n}}|\nabla w|^2~dx\right)^\frac{1}{2},
\end{equation}
which for $d=2$ takes on the form
\begin{equation}\label{a.4}
\left(\dashint_{B_{2^n}}(w-\dashint_{B_{2^n}}w)^2~dx\right)^\frac{1}{2}
\leq C\left(\int_{B_{2^n}}|\nabla w|^2~dx\right)^\frac{1}{2}
\stackrel{(\ref{a.3})}{\leq}C(\lambda).
\end{equation}
By Jensen's and the triangle inequality, this yields in particular
$|\dashint_{B_{2^{n-1}}}w~dx-\dashint_{B_{2^{n}}}w~dx|\leq C(\lambda)$ and
thus, since we may w.\ l.\ o.\ g.\ assume $\int_{B_1}w~dx=0$,
$|\dashint_{B_{2^{n}}}w~dx|\leq n C(\lambda)$.
Inserting this back into (\ref{a.4}) gives
\begin{equation}\nonumber
\left(\dashint_{B_{2^n}}w^2~dx\right)^\frac{1}{2}\leq n C(\lambda),
\end{equation}
that is, (\ref{a.6}) in the stronger form of
\begin{equation}\nonumber
\Big(\dashint_{B_R}w^2~dx\Big)^\frac{1}{2}\leq C(d)\log R.
\end{equation}
\end{proof}

\bibliographystyle{plain}
\bibliography{stochastic_homogenization}

\begin{thebibliography}{10}

\bibitem{ArmstrongMourrat}
Scott Armstrong and Jean-Christophe Mourrat.
\newblock Lipschitz regularity for elliptic equations with random coefficients.
\newblock {\em to appear in Arch. Ration. Mech. Anal.}, 2014.
\newblock arXiv:1411.3668.

\bibitem{ArmstrongSmart}
Scott Armstrong and Charles~K. Smart.
\newblock Quantitative stochastic homogenization of convex integral
  functionals.
\newblock {\em to appear in Ann. Sci. \'Ec. Norm. Sup\'er}, 2014.
\newblock arXiv:1406.0996.

\bibitem{AvellanedaLin}
M.~Avellaneda and F.H. Lin.
\newblock Une th\'eor\`eme de liouville pour des \'equations elliptiques \`a
  coefficients p\'eriodiques.
\newblock {\em C. R. Acad. Sci. Paris S\'er. I Math.}, 309:245--250, 1989.

\bibitem{BenjaminiCopinKozmaYadin}
Itai Benjamini, Hugo Duminil-Copin, Gady Kozma, and Ariel Yadin.
\newblock Disorder, entropy and harmonic functions.
\newblock {\em to appear in Ann. Probab.}, 2011.
\newblock arXiv:1111.4853.

\bibitem{FischerOtto2}
Julian Fischer and Felix Otto.
\newblock Sublinear growth of the corrector in stochastic homogenization:
  Optimal stochastic estimates for slowly decaying correlations.
\newblock {\em Preprint}, 2015.
\newblock arXiv:1508.00025.

\bibitem{GiaquintaMartinazzi}
Mariano Giaquinta and Luca Martinazzi.
\newblock {\em An introduction to the regularity theory for elliptic systems,
  harmonic maps and minimal graphs}, volume~11 of {\em Appunti. Scuola Normale
  Superiore di Pisa (Nuova Serie) [Lecture Notes. Scuola Normale Superiore di
  Pisa (New Series)]}.
\newblock Edizioni della Normale, Pisa, second edition, 2012.

\bibitem{GloriaNeukammOtto}
Antoine Gloria, Stefan Neukamm, and Felix Otto.
\newblock A regularity theory for random elliptic operators.
\newblock {\em Preprint}, 2014.
\newblock arXiv:1409.2678.

\bibitem{GloriaOttoAP2011}
Antoine Gloria and Felix Otto.
\newblock An optimal variance estimate in stochastic homogenization of discrete
  elliptic equations.
\newblock {\em Ann. Probab.}, 39(3):779--856, 2011.

\bibitem{MarahrensOtto}
Daniel Marahrens and Felix Otto.
\newblock Annealed estimates on the {G}reen function.
\newblock {\em to appear in Probab. Theory Related Fields}, 2013.
\newblock arXiv:1304.4408.

\bibitem{PiccininiSpagnolo}
L.~C. Piccinini and S.~Spagnolo.
\newblock On the {H}\"older continuity of solutions of second order elliptic
  equations in two variables.
\newblock {\em Ann. Scuola Norm. Sup. Pisa (3)}, 26:391--402, 1972.

\end{thebibliography}

\end{document}